\documentclass[9pt,reqno]{amsart}

\usepackage[utf8]{inputenc}
\usepackage[english]{babel}
\usepackage{amsmath,amsfonts,amssymb,amsthm,shuffle}
\usepackage[T1]{fontenc}
\usepackage{lmodern}
\usepackage{mathtools}
\usepackage[titletoc]{appendix}
\usepackage{enumitem}

\usepackage[top=3.5cm,bottom=3.5cm,left=3.6cm,right=3.6cm]{geometry}

\usepackage[dvipsnames]{xcolor}
\usepackage[hyperindex=true,frenchlinks=true,colorlinks=true,
citecolor=Mahogany,linkcolor=DarkOrchid,urlcolor=Tan,linktocpage,
pagebackref=true]{hyperref}

\usepackage{tikz}
\usetikzlibrary{shapes}
\usetikzlibrary{fit}
\usetikzlibrary{decorations.pathmorphing}

\usepackage{dsfont}
\usepackage{wasysym}
\usepackage{stmaryrd}
\usepackage{cite}
\usepackage{subfigure}
\usepackage{multirow}
\usepackage{enumitem}
\usepackage{multicol}
\usepackage{amsaddr}

\title[]{Convergence of the empirical
two-sample $U$-statistics with $\beta$-mixing data}
 
\date{\today}

\numberwithin{equation}{section}
\renewcommand{\leq}{\leqslant}
\renewcommand{\geq}{\geqslant}

\newtheorem{Theorem}{Theorem}[section]
\newtheorem{Th\'eor\`eme}{Th\'eor\`eme}[section]
\newtheorem{Proposition}[Theorem]{Proposition}
\newtheorem{Lemma}[Theorem]{Lemma}

\newtheorem{D\'efinition}[Th\'eor\`eme]{D\'efinition}
\newtheorem{Corollary}[Theorem]{Corollary}

\theoremstyle{remark}
\newtheorem{Remark}[Theorem]{Remark}

\tikzstyle{Vertex}=[circle,draw=LimeGreen!80,fill=LimeGreen!8,
inner sep=1pt,minimum size=2mm,line width=1pt,font=\scriptsize]
\tikzstyle{Node}=[Vertex,draw=RoyalBlue!80,fill=RoyalBlue!8,inner sep=1.5pt]
\tikzstyle{Leaf}=[rectangle,draw=Black!70,fill=Black!16,
inner sep=0pt,minimum size=1mm,line width=1.25pt]
\tikzstyle{Edge}=[Maroon!80,cap=round,line width=1pt]
\tikzstyle{Mark1}=[draw=BrickRed!80,fill=BrickRed!8]
\tikzstyle{Mark2}=[draw=BurntOrange!80,fill=BurntOrange!8]
\tikzstyle{EdgeRew}=[->,RedOrange!80,cap=round,thick]

\newcommand{\Aca}{\mathcal{A}}
\newcommand{\Bca}{\mathcal{B}}

\newcommand{\Fca}{\mathcal{F}}

\newcommand \ens[1]{\left\{ #1\right\}}
\newcommand \R{\mathbb R}

\newcommand \N{\mathbb N}

\newcommand \PP{\mathbb P}

\newcommand{\E}[1]{\mathbb E\left[#1\right]}

\newcommand \Z{\mathbb Z}

\newcommand \abs[1]{\left|#1\right|}
\newcommand \eps{\varepsilon}

\newcommand{\pr}[1]{\left(#1\right)}
\newcommand{\norm}[1]{\left\lVert #1 \right\rVert}

\newcommand{\ent}[1]{\left\lfloor #1\right\rfloor}

\newcommand{\1}[1]{\mathbf{1}{\ens{#1}}}

\newcommand{\cov}[2]{\operatorname{Cov}\pr{#1,#2}}

\author{Herold Dehling, Davide Giraudo and Olimjon Sharipov}
\address{ Herold.Dehling@rub.de Davide.Giraudo@rub.de \\   Ruhr-Universität Bochum, Germany\\
 osharipov@yahoo.com\\
 Uzbek Academy of Sciences
Department of Probability Theory and Mathematical Statistics, 
Tashkent, Uzbekistan}
\begin{document}

\begin{abstract}
We consider the empirical two-sample $U$-statistic with 
strictly $\beta$-mixing strictly stationary
data and inverstigate its convergence in Skorohod spaces. We then 
provide an application of such convergence. 

\end{abstract}

\maketitle 

\section{Introduction and main results}


In this paper, we investigate the large sample behavior of the  empirical distribution function of the data
$g\pr{X_i,X_j}$, $1\leq i\leq [nt]$, $[nt]+1 \leq j\leq n$, indexed by the parameter  $0\leq t\leq 1$, where $\pr{X_i}_{i\geq 1}$ is a stochastic process, and where $g:\R^2\rightarrow \R$ is some given measurable  function.  This so-called two-sample $U$-statistic empirical distribution function arises in the context of robust tests for change-points in time series. For any $t\in [0,1]$, we subdivide the data $X_1,\ldots, X_n$ into two samples $X_1,\ldots,X_{[nt]}$ and $X_{[nt]+1},\ldots, X_n$, and perform a two-sample test for equality of some relevant parameter against the alternative of a change.  E.g., the Wilcoxon change-point test is based on the relative frequency of pairs of observations $\pr{X_i,X_j}$, with $1\leq i\leq [nt]$, $[nt]+1 \leq j\leq n$, such that $X_i-X_j \leq 0 $.  The Hodges-Lehmann change-point test, introduced by Dehling, Fried and Wendler \cite{10.1093/biomet/asaa004}, is based on the median of the data $X_i-X_j$, and thus on quantiles of the empirical distribution of $X_i-X_j$, $1\leq i\leq [nt]$, $[nt]+1 \leq j\leq n$.

We will analyze the asymptotic distribution of the centered and normalized empirical distribution function of the data  $g(X_i,X_j)$, $1\leq i\leq [nt]$, $[nt]+1 \leq j\leq n$, given by
\begin{equation}
e_n\pr{s,t}:= \frac{1}{n^{3/2}}\sum_{i=1}^{[nt]}\sum_{j=[nt]+1}^n 
\pr{\mathbf 1\ens{g\pr{X_i,X_j}\leq s}-\PP\ens{g\pr{X_i,X_j}\leq s}  },  0\leq t\leq 1, s\in \R,
\label{eq:2-s-emp-u}
\end{equation}
and viewed as a two-parameter process indexed by the parameters $(s,t)\in \R\times [0,1]$. We call the process $(e_n(s,t))_{(s,t)\in \R\times [0,1]}$ a two-sample empirical $U$-process, and we will show that this process converges, as $n\rightarrow \infty$, to a Gaussian limit process, if the underlying process $(X_i)_{i\geq 1}$ is short range dependent. 

Related processes have been studied in the literature. Dehling, Fried, Garcia and Wendler \cite{MR3409833} have analyzed the large sample behavior of the two-sample $U$-process, defined by
\[
  \frac{1}{n^{3/2}} \sum_{i=1}^{[nt]} \sum_{j=[nt]+1}^n h(X_i,X_j), \; 0\leq t\leq 1,
\]
where  $h\colon \R^2\rightarrow \R$ is a measurable function, also called kernel of the $U$-statistic. 
Note that, for fixed value of $s\in \R$, the two-sample empirical $U$-process is a two-sample $U$-process, with the indicator kernel $h(x,y)=\mathbf{1}\ens{g(x,y)\leq s }$. Dehling et al. 
\cite{MR3409833} proved convergence of this process to a Gaussian limit process, in case the underlying data are functionals of an absolutely regular process, thus extending earlier results by Cs\"org\H{o} and Horv\'ath \cite{MR971179} for i.i.d. data. 

Along another line, various authors have investigated the large sample behavior of the one-sample empirical $U$-process, defined as 
\[
  \frac{1}{n^{3/2}} \sum_{1\leq i<j\leq n} \pr{ 
\mathbf{1}\ens{ g\pr{X_i,X_j} \leq s   } -
\PP\pr{g\pr{X_1,X_2}\leq s}},\, s\in \R,
\]
where $g\colon \R^2\rightarrow \R$ is a symmetric kernel.
Serfling \cite{MR733500} initiated the study of these processes for i.i.d. data in connection with so-called generalized $L$-. $M$-, and $R$-statistics. Dehling, Denker and Philipp \cite{MR891707} proved an almost sure invariance principle, again for i.i.d. data. Arcones and Yu \cite{MR1256391} investigated the one-sample empirical $U$-process for absolutely regular data.  Motivated by applications to estimation of the correlation dimension of chaotic dynamical systems, Borovkova, Burton and Dehling \cite{MR1851171} showed weak convergence of the one-sample empirical $U$-process to a Gaussian limit process, when the underlying data are functionals of absolutely regular processes. L\'evy-Leduc, Boistard, Moulines, Reisen and Taqqu  
\cite{MR2850207} studied the empirical $U$-process under long-range dependence. 

In the present paper, we focus on the two-parameter empirical
$U$-process, defined in \eqref{eq:2-s-emp-u}, when the underlying data are generated by a stationary mixing stochastic process $(X_i)_{i\geq 1}$.  From the results of Dehling et al. \cite{MR3409833}, we obtain convergence of the process $\pr{e_n\pr{s_0,t}}_{0\leq t\leq 1}$ for a fixed $s_0$. Here, we will extend this result to the two-parameter process $\pr{e_n(s,t)}_{\pr{s,t}\in \R\times [0,1]}$. 


In this paper, we focus on the case of mixing sequences.
Let $\pr{\Omega,\Fca,\PP}$ be a probability space. 
The $\alpha$-mixing and $\beta$-mixing coefficients between two sub-$\sigma$-algebras 
$\Aca$ and $\Bca$ of $\Fca$ are defined defined respectively by 

\begin{equation}
 \alpha\pr{\Aca,\Bca}=\sup\ens{
 \abs{\PP\pr{A\cap B}-\PP\pr{A}\PP\pr{B}},A\in\Aca, B\in \Bca
 };
\end{equation}
\begin{equation}
 \beta\pr{\Aca,\Bca}=\frac 12\sup\ens{
 \sum_{i=1}^I\sum_{j=1}^J\abs{\PP\pr{A_i\cap B_j}-
 \PP\pr{A_i}\PP\pr{B_j}}},
\end{equation}
where the supremum runs over all the partitions 
$\pr{A_i}_{i=1}^I$ and $\pr{B_j}_{j=1}^J$ of $\Omega$ of 
elements of $\Aca$ and $\Bca$ respectively.

Given a sequence $\pr{X_i}_{i\geq 1}$, we associate its 
sequences of $\alpha$ and $\beta$-mixing coefficients 
by letting 
\begin{equation}
\alpha\pr{k}:=\sup_{\ell\geq 1}
\alpha\pr{\Fca_1^\ell,\Fca_{\ell+k}^{\infty}},
\end{equation}
\begin{equation}
\beta\pr{k}:=\sup_{\ell\geq 1}
\beta\pr{\Fca_1^\ell,\Fca_{\ell+k}^{\infty}},
\end{equation}
where $\Fca_u^v$, $1\leq u\leq v\leq +\infty$ is 
the $\sigma$-algebra generated by the random variables 
$X_i$, $u\leq i\leq v$ ($u\leq i$ for $v=+\infty$).

\subsection{Convergence of the two-sample $U$-statistic in Skorohod spaces $D\pr{[-R,R]\times [0,1]}$}

Let us state one of the two main results of the paper.

\begin{Theorem}\label{thm:convergence_two_sample_emp_Ustat}
Let $\pr{X_i}_{i\in\Z}$ be a strictly stationary sequence. Let $e_n$ be the two-sample $U$-statistics empirical process 
with kernel $g\colon \R\times\R\to \R$ defined 
for $n\geq 1$, $0\leq t\leq 1$ and $s\in \R$ by 
\begin{equation}
e_n\pr{s,t}:= \frac{1}{n^{3/2}}\sum_{i=1}^{[nt]}\sum_{j=[nt]+1}^n 
\pr{\mathbf 1\ens{g\pr{X_i,X_j}\leq s}-\PP\ens{g\pr{X_i,X_j}\leq s}  }   .
\end{equation}
Suppose that the following four conditions holds.
\begin{enumerate}[label=(A.\arabic*)]
 \item\label{asp:bounded_density_coord_1} For all $u\in \R$, the random variable $g\pr{u,X_1}$ 
 has a density $f_{1,u}$ and $\sup_{x,u\in\R}f_{1,u}\pr{x}<+\infty$.
 \item\label{asp:bounded_density_coord_2} For all $v\in\R$, the random variable $g\pr{X_1,v}$ 
 has a density $f_{2,v}$ and $\sup_{x,v\in\R}
 f_{2,v}\pr{x}<+\infty$.
 \item\label{asp:alpha_mixing} There exists a $p>2$ such that 
 $\sum_{k\geq 1}k^p\alpha\pr{k}$ converges.
 \item \label{asp:beta_mixing} The series $\sum_{k\geq 1}
 k\beta\pr{k}$ converges.
\end{enumerate}
 
Then for all $R$, 
\begin{equation}
e_n\pr{s,t}\to W\pr{s,t}\mbox{ in distribution in }D\pr{[-R,R]\times [0,1]},
\end{equation}
where $\pr{W\pr{s,t}, s\in \R,t\in [0,1]}$ is a centered 
Gaussian process, with covariance given for $0\leq t\leq t'\leq 1$ and $s,s'\in \R$ by 
the following formula:
\begin{multline}
\operatorname{Cov}\pr{W\pr{s,t},W\pr{s',t'}}=
t\pr{1-t}\pr{1-t'}
C_{1,1}\pr{s,s'}\\
+t\pr{1-t'}\pr{t'-t}C_{2,1}\pr{s,s'}
+tt'\pr{1-t'}C_{2,2}\pr{s,s'},
\end{multline}
where for $i,j=1;2$ and $s,s'\in \R$, 
\begin{equation}
 C_{i,j}\pr{s,s'}= 
 \sum_{k\in \Z}\operatorname{Cov}\pr{h_{i,s}\pr{X_0}, h_{j,s'}\pr{X_k}},
\end{equation}

\begin{equation}
h_{1,s}\pr{u}= \PP\ens{g\pr{u,X_1}\leq s}, 
\end{equation}
 \begin{equation}
   h_{2,s}\pr{v}= 
\PP\ens{g\pr{ X_1,v}\leq s}.
 \end{equation}
\end{Theorem}

\begin{Remark}\label{rem:symmetry}
We did not make a symmetry assumption on $g$. When $g$ is symmetric, in the sense that 
$g\pr{u,v}=g\pr{v,u}$ for all $u$ and $v\in \R$, the 
covariance of the limiting process $W$ reads 
\begin{equation}
\operatorname{Cov}\pr{W\pr{s,t},W\pr{s',t'}}=
t\pr{1-t'}\pr{1+2t'-2t}C_{1,1}\pr{s,s'}.
\end{equation}
\end{Remark}

In practical cases, the probability $\PP\ens{g\pr{X_i,X_j}\leq s}$ is 
unknown, and we only have the values of $\mathbf 1\ens{g\pr{X_i,X_j}\leq s}$, $1\leq i<j\leq n$, 
at our disposal. This leads to an analoguous result as Theorem~\ref{thm:convergence_two_sample_emp_Ustat}, 
where the quantity $\PP\ens{g\pr{X_i,X_j}\leq s}$ is replaced by its empirical estimator 
$ {{n}\choose{2}}^{-1}\sum_{1\leq i<j\leq n}\mathbf 1\ens{g\pr{X_i,X_j}\leq s}$.

\begin{Theorem}
\label{thm:convergence_two_sample_emp_Ustat_prob_replaced}
Let $\pr{X_i}_{i\in\Z}$ be a strictly stationary sequence. Let $e'_n$ be the two-sample $U$-statistics empirical process 
with kernel $g\colon \R\times\R\to \R$ defined 
for $n\geq 1$, $0\leq t\leq 1$ and $s\in \R$ by 
\begin{equation}
e'_n\pr{s,t}:= \frac{1}{n^{3/2}}\sum_{i=1}^{[nt]}\sum_{j=[nt]+1}^n 
\pr{\mathbf 1\ens{g\pr{X_i,X_j}\leq s}-\frac1{{{n}\choose{2}}}\sum_{1\leq i'<j'\leq n}\mathbf 1\ens{g\pr{X_{i'},X_{j'}}\leq s} }.
\end{equation}
Suppose that the assumptions \ref{asp:bounded_density_coord_1}, 
\ref{asp:bounded_density_coord_2}, \ref{asp:alpha_mixing} and 
\ref{asp:beta_mixing} hold.
Then for all $R$, 
\begin{equation}
e'_n\pr{s,t}\to W\pr{s,t}\mbox{ in distribution in }D\pr{[-R,R]\times [0,1]},
\end{equation}
where $\pr{W\pr{s,t}, s\in \R,t\in [0,1]}$ is a centered 
Gaussian process, with covariance given for $0\leq t\leq t'\leq 1$ and $s,s'\in \R$ by 
the following formula:
\begin{multline}\label{eq:expre_cov_oper_centrage_emp}
\operatorname{Cov}\pr{W\pr{s,t},W\pr{s',t'}}=
t\pr{1-t}\pr{1-t'}C_{1,1}\pr{s,s'}  
+\pr{t'-t}t\pr{1-t'}C_{2,1}\pr{s,s'} \\
+ \pr{1-t'}tt'C_{2,2}\pr{s,s'}
-2tt'\pr{1-t^2}\pr{1-t'}C_1^a\pr{s,s'}\\
-2t\pr{1-t}t'\pr{1-t'}\pr{s',s}C_1^a\pr{s',s}
-2tt'\pr{1-t'}\pr{2+t'-2t-t^2}C_2^a\pr{s,s'}\\
-2t\pr{1-t}t'\pr{1-t'}C_2^a\pr{s',s}
+4t\pr{1-t}t'\pr{1-t'}\pr{t^2+t+\frac{10}3}C^a\pr{s,s'},
 \end{multline}
where 
 
\begin{equation}
h_{1,s}\pr{u}= \PP\ens{g\pr{u,X_1}\leq s}, 
\end{equation}
 \begin{equation}
   h_{2,s}\pr{v}= 
\PP\ens{g\pr{ X_1,v}\leq s},
 \end{equation}
 \begin{equation}
  a_s=h_{1,s}-h_{2,s},  
 \end{equation}
 for $i,j=1;2$,
\begin{equation}
 C_{i,j}\pr{s,s'}=\sum_{k\in\Z}
 \cov{h_{i,s}\pr{X_0}}{h_{j,s'}\pr{X_k}}, 
\end{equation}
\begin{equation}
 C_{i}^a\pr{s,s'}=\sum_{k\in\Z}
 \cov{h_{i,s}\pr{X_0}}{a_{s'}\pr{X_k}}
\end{equation}
\begin{equation}
 C^a\pr{s,s'}=\sum_{k\in\Z}
 \cov{a_{s}\pr{X_0}}{a_{s'}\pr{X_k}}.
\end{equation}
\end{Theorem}
\begin{Remark}
When $g$ is symmetric, $h_{1,s}=h_{2,s}$ and $a_s=0$ hence 
the covariance admits the simpler form 
\begin{equation}
\operatorname{Cov}\pr{W\pr{s,t},W\pr{s',t'}}=
t\pr{1-t'}\pr{1-2t+2t'}C_{1,1}\pr{s,s'}, t\leq t',s\in \R,
\end{equation}
in particular, we get the same limiting process as in 
the centering of the indicator by their expectation
(see Remark~\ref{rem:symmetry}).
\end{Remark}

Let us give examples where the assumptions~\ref{asp:bounded_density_coord_1} and 
\ref{asp:bounded_density_coord_2} are satisfied. Let $g_1$ and $g_2$ be function 
defined from $\R$ to itself. Assume that $g_1\pr{X_1}$ has a density $f_1$ and 
$g_2\pr{X_1}$ has a density $f_2$, where $f_1$ and $f_2$ are bounded.
\begin{enumerate}
\item Let $g\colon \pr{u,v}\mapsto g_1\pr{u}+g_2\pr{v}$. Then $g\pr{u,X_1}$ has 
density $f_{1,u}$ where $f_{1,u}\pr{x}=  f_2\pr{x-g_1\pr{u}}$ hence 
$\sup_{x,u\in\R}f_{1,u}\pr{x}=\sup_{x\in \R}f_2\pr{x}<+\infty$ and similarly, 
$\sup_{x,u\in\R}f_{2,u}\pr{x}=\sup_{x\in \R}f_1\pr{x}<+\infty$.
\item Let $g\colon \pr{u,v}\mapsto \abs{g_1\pr{u}-g_2\pr{v}}$. Then $f_{1,u}\pr{x}=0$ for $x<0$ and for $x\geq 0$, 
\begin{equation}
f_{1,u}\pr{x}= f_2\pr{g_1\pr{u}+x}+f_2\pr{g_1\pr{u}-x}
\end{equation}
hence 
\begin{equation}
\sup_{x,u\in\R}f_{1,u}\pr{x}\leq \sup_{x,u\in\R}f_2\pr{g_1\pr{u}+x}+\sup_{x,u\in\R}f_2\pr{g_1\pr{u}-x}
\leq 2\sup_{x\in \R}f_2\pr{x}<+\infty
\end{equation}
and by a similar reasoning, we also derive that
\begin{equation}
\sup_{x,u\in\R}f_{2,u}\pr{x} 
\leq 2\sup_{x\in \R}f_1\pr{x}<+\infty.
\end{equation}
\end{enumerate}

\subsection{Application}
 
The following corollaries are a consequence of Theorems~\ref{thm:convergence_two_sample_emp_Ustat} and 
\ref{thm:convergence_two_sample_emp_Ustat_prob_replaced}.

\begin{Corollary}\label{cor:supremum}
 Under the conditions of Theorems~\ref{thm:convergence_two_sample_emp_Ustat} and 
\ref{thm:convergence_two_sample_emp_Ustat_prob_replaced} the following convergences in distribution take 
place for all positive $R$:
\begin{equation}\label{eq:conv_sup_pour_en}
 \sup_{0\leq t\leq 1}\sup_{-R\leq s\leq R}
  \abs{ e_n\pr{s,t}} \to 
  \sup_{0\leq t\leq 1}\sup_{-R\leq s\leq R}
  \abs{W\pr{s,t}};
\end{equation}
\begin{equation}\label{eq:conv_sup_pour_e'n}
\sup_{0\leq t\leq 1}\sup_{-R\leq s\leq R}
  \abs{ e'_n\pr{s,t}}   \to 
  \sup_{0\leq t\leq 1}\sup_{-R\leq s\leq R}
  \abs{W'\pr{s,t}}.
\end{equation}
\end{Corollary}

\begin{Corollary}\label{cor:application}
  Let $\mu$ be a finite measure on the Borel subsets of $\R$. 
Then under the conditions of Theorems~\ref{thm:convergence_two_sample_emp_Ustat} and 
\ref{thm:convergence_two_sample_emp_Ustat_prob_replaced} the following convergences in distribution take 
place
\begin{equation}\label{eq:conv_fonctionnelle_pour_en}
 \sup_{0\leq t\leq 1}\int_{\R}
   e_n\pr{s,t}  ^2d\mu\pr{s}\to 
  \sup_{0\leq t\leq 1}\int_{\R}
  W\pr{s,t}^2d\mu\pr{s}
\end{equation}
\begin{equation}\label{eq:conv_fonctionnelle_pour_e'n}
 \sup_{0\leq t\leq 1}\int_{\R}
   e'_n\pr{s,t}  ^2d\mu\pr{s}\to 
  \sup_{0\leq t\leq 1}\int_{\R}
  W'\pr{s,t}^2d\mu\pr{s}.
\end{equation}
\end{Corollary}

\section{Proof}

The proof of Theorems~\ref{thm:convergence_two_sample_emp_Ustat} and 
\ref{thm:convergence_two_sample_emp_Ustat_prob_replaced} 
will be done according to the following steps.
\begin{enumerate}
 \item Let $h_s\colon\R^2\to\R$ be the kernel defined by 
 $h_s\pr{u,v}=\1{g\pr{u,v}\leq s}$. The Hoeffding's 
 decomposition of this kernel gives a spliting of the 
 empirical two-sample $U$-statistics into a linear part and a degenerated part.
 \item We prove the convergence of the finite dimensional 
 distributions of the linear part to the corresponding ones 
 of the process $W$.
 \item Then we prove that the process associated to the linear 
 part converges to $W$ in $D\pr{[-R,R]\times [0,1]}$ for all 
 $R>0$.
 \item Finally, we show the negligibility of 
 the contribution of the degenerated part.
\end{enumerate}
We do it first in the context of Theorem~\ref{thm:convergence_two_sample_emp_Ustat}. The proof of 
Theorem~\ref{thm:convergence_two_sample_emp_Ustat_prob_replaced} is closely related. Consequently, 
we will only mention the required modifications.

\subsection{Proof of Theorem~\ref{thm:convergence_two_sample_emp_Ustat}}
\subsubsection{Hoeffding's decomposition}

Let $h_s\colon \R^2\to \R$ be defined by $h_s\pr{u,v}=\mathbf{1}\ens{g\pr{u,v}\leq s}$. Let us do the 
Hoeffding's decomposition of $h_s$ for each fixed $s$. Let $\theta_s:=\PP\ens{g\pr{X_1,X'_1}\leq s}$, 
where $X'_1$ is an independent copy of $X_1$,
\begin{equation}\label{eq:def_noyau_lin}
h_{1,s}\pr{u}= \PP\ens{g\pr{u,X_1}\leq s}-\theta_s, \quad h_{2,s}\pr{v}= \PP\ens{g\pr{ X_1,v}\leq s}-\theta_s,
\end{equation}
 and 
 \begin{equation}\label{eq:def_noyau_deg}
 h_{3,s}\pr{u,v}=h_s\pr{u,v}-h_{1,s}\pr{u}-h_{2,s}\pr{v}-\theta_s.
 \end{equation}
 Then 
 \begin{align*}
 e_n\pr{s,t}&=\frac 1{n^{3/2}}\sum_{i=1}^{[nt]}\sum_{j=[nt]+1}^n 
\pr{h_s\pr{X_i,X_j}-\E{h_s\pr{X_i,X_j}}}  \\ 
&=\frac 1{n^{3/2}}\sum_{i=1}^{[nt]}\sum_{j=[nt]+1}^n 
 \pr{h_{3,s}\pr{X_i,X_j}-\E{h_{3,s}\pr{X_i,X_j}}}  + \frac 1{n^{3/2}}\sum_{i=1}^{[nt]}\sum_{j=[nt]+1}^n 
 \pr{h_{1,s}\pr{X_i}-\E{h_{1,s}\pr{X_i}}}\\
 &+\frac 1{n^{3/2}}\sum_{i=1}^{[nt]}\sum_{j=[nt]+1}^n 
 \pr{h_{2,s}\pr{X_j}-\E{h_{2,s}\pr{X_j}}}  \\
 &=\frac 1{n^{3/2}}\sum_{i=1}^{[nt]}\sum_{j=[nt]+1}^n 
 \pr{h_{3,s}\pr{X_i,X_j}-\E{h_{3,s}\pr{X_i,X_j}}} + \frac{n-[nt]}{n^{3/2}}\sum_{i=1}^{[nt]} 
  \pr{h_{1,s}\pr{X_i}-\E{h_{1,s}\pr{X_i}}}\\
  &+\frac{[nt]}{n^{3/2}} \sum_{j=[nt]+1}^n 
\pr{h_{2,s}\pr{X_j}-\E{h_{2,s}\pr{X_j}}}
 \end{align*}
 or in other words, 
 \begin{equation}\label{eq:decomposition_de_e_n}
  e_n\pr{s,t}=R_n\pr{s,t}+W_n\pr{s,t},
 \end{equation}
 where 
 \begin{equation}\label{eq:definition_de_Rn}
 R_n\pr{s,t}=\frac 1{n^{3/2}}\sum_{i=1}^{[nt]}\sum_{j=[nt]+1}^n 
 \pr{h_{3,s}\pr{X_i,X_j}-\E{h_{3,s}\pr{X_i,X_j}}} ,
 \end{equation}
\begin{equation}\label{eq:definition_de_Wn}
W_n\pr{s,t}=\frac{n-[nt]}{n^{3/2}}\sum_{i=1}^{[nt]} 
\pr{h_{1,s}\pr{X_i}-\E{h_{1,s}\pr{X_i}}}+\frac{[nt]}{n^{3/2}} \sum_{j=[nt]+1}^n \pr{h_{2,s}\pr{X_j}-\E{h_{2,s}\pr{X_j}}}.
\end{equation}

Moreover, observe that by the assumpions~\ref{asp:bounded_density_coord_1} and 
\ref{asp:bounded_density_coord_2}, there exists a constant $M$ such that 
for all $i\in \ens{1;2;3}$,
\begin{equation}\label{eq:Lipschitz_en_norme_unif}
\sup_{x\in \R}\sup_{s<s'}\frac{\abs{h_{i,s'}\pr{x} -h_{i,s}\pr{x}  }}{s'-s}\leq M.
\end{equation}

\subsubsection{Convergence of the finite dimensional distributions 
of the linear part}

\begin{Proposition}\label{prop:conv_fidi}
For all $d\geq 1$ and all $s_1<\dots<s_d$ and $0\leq t_1
<\dots<t_d\leq 1$, the vector $\pr{W_n\pr{s_\ell,t_k}}_{k,\ell
=1}^d$ converges in distribution to $\pr{W\pr{s_\ell,t_k}}_{k,\ell
=1}^d$, where $W_n$ is defined by \eqref{eq:definition_de_Wn} and 
$W$ is like in Theorem~\ref{thm:convergence_two_sample_emp_Ustat}.
\end{Proposition}

\begin{proof}
  We will use the Cramer-Wold device. Let $\pr{a_{k,\ell}}_{k,\ell 
 =1}^d$ be a family of real numbers. We have to prove that 
 \begin{equation}
  \sum_{k,\ell=1}^d a_{k,\ell}W _n\pr{s_\ell,t_k}
  \to \sum_{k,\ell=1}^d a_{k,\ell}W\pr{s_\ell,t_k}\mbox{ in 
  distribution}.
 \end{equation}
To this aim, we will express $\sum_{k,\ell=1}^d
a_{k,\ell}W _n\pr{s_\ell,t_k}$ as a sum of linear combinations 
of a mixing sequence random variables, and then apply a 
central limit theorem, namely, Theorem~\ref{thm:TLC_mixing_array}.  Let
\begin{equation}\label{eq:def_of_Inu}
 I_{n,u}=\ens{i\in \N\mid [nt_{u-1}]+1\leq i\leq 
 [nt_u]},  2\leq u\leq d,
\end{equation}
$I_{n,1}=\ens{i\in \N\mid 1\leq i\leq [nt_1]}$ and 
 $I_{n,d+1}=\ens{i\in \N\mid [nt_{d}]+1\leq i\leq 
n}$.  Then the following equality holds:
\begin{multline}
 \sum_{k,\ell=1}^d a_{k,\ell}W _n\pr{s_\ell,t_k}= 
 \sum_{k,\ell=1}^d a_{k,\ell} 
 \frac{n-[nt_k]}{ n^{3/2}}\sum_{u=1}^{k} 
 \sum_{i\in I_{n,u} }\pr{h_{1,s_\ell}\pr{X_i}-\E{h_{1,s_\ell}\pr{X_i}}}\\ +
  \sum_{k,\ell=1}^d a_{k,\ell} \frac{[nt_k]}{n^{3/2}}\sum_{u=k+1}^{d+1}  
 \sum_{i\in I_{n,u} }\pr{h_{2,s_\ell}\pr{X_i}-\E{h_{2,s_\ell}\pr{X_i}}},
\end{multline}
or in other words, 
\begin{multline}
  \sum_{k,\ell=1}^d a_{k,\ell}W_n\pr{s_\ell,t_k}= 
  \sum_{u=1}^{d+1} \sum_{i\in I_{n,u} } \sum_{k,\ell=1}^d a_{k,\ell}
 \frac{n-[nt_k]}{ n^{3/2}} \mathbf 1\ens{u\leq k}\pr{
h_{1,s_\ell}\pr{X_i}-\E{h_{1,s_\ell}\pr{X_i}}} \\
+\sum_{u=1}^{d+1} \sum_{i\in I_{n,u} } \sum_{k,\ell=1}^d a_{k,\ell}
 \frac{[nt_k]}{n^{3/2}} \mathbf 1\ens{u\geq k+1}
 \pr{h_{2,s_\ell}\pr{X_i}-\E{h_{2,s_\ell}\pr{X_i}}}.
\end{multline}
Defining for $i\in I_{n,u},1\leq u\leq d+1$ the random variable 
$Y_{n,i}$ by 
\begin{multline}
 Y_{n,i}:= 
\sum_{k,\ell=1}^d a_{k,\ell}
 \frac{n-[nt_k]}{ n^{3/2}} \mathbf 1\ens{u\leq k}
 \pr{h_{1,s_\ell}\pr{X_i}-\E{h_{1,s_\ell}\pr{X_i}}} \\
+ \sum_{k,\ell=1}^d a_{k,\ell}
 \frac{[nt_k]}{n^{3/2}} \mathbf 1\ens{u\geq k+1}
 \pr{h_{2,s_\ell}\pr{X_i}-\E{h_{2,s_\ell}\pr{X_i}}},  
\end{multline}
it follows that $\sum_{k,\ell=1}^d a_{k,\ell}W_n\pr{s_\ell,t_k}= 
n^{-1/2}\sum_{i=1}^nY_{n,i}$. Observe also that 
$\E{Y_{n,i}}=0$. We want to check the conditions of 
Theorem~\ref{thm:TLC_mixing_array}. The first 
condition follows from 
\begin{multline}
 \abs{Y_{n,i}}\leq \sum_{k,\ell=1}^d \abs{a_{k,\ell}}
 \pr{\mathbf 1\ens{u\leq k}
\abs{h_{1,s_\ell}\pr{X_i}-\E{h_{1,s_\ell}\pr{X_i}}}}\\
+\sum_{k,\ell=1}^d \abs{a_{k,\ell}}\mathbf 1\ens{u\geq k+1}
 \abs{ h_{2,s_\ell}\pr{X_i}-\E{h_{2,s_\ell}\pr{X_i}}}\leq 2\sum_{k,\ell=1}^d \abs{a_{k,\ell}}.
\end{multline}
For the second condition, we first observe that if $
1\leq u< u'\leq d+1$, then 
\begin{equation}
\label{eq:cov_blocs_disjoints}
 \frac 1n\cov{\sum_{i\in I_{n,u}}Y_{n,i}   }{\sum_{i'\in I_{n,u'}}Y_{n,i}}\to 0.
\end{equation}
Indeed, by Proposition~\ref{prop:cov_inegalite}, 
for $i\in I_{n,u}$ and $i'\in I_{n,u'}$, 
\begin{equation}
 \abs{\cov{Y_{n,i}}{Y_{n,i'} }}\leq 
 \alpha\pr{i'-i}\pr{2\sum_{k,\ell=1}^d \abs{a_{k,\ell}}}^2
 = :\alpha\pr{i'-i}K
\end{equation}
hence 
\begin{equation}
 \frac 1n
 \abs{\cov{\sum_{i\in I_{n,u}}Y_{n,i}}{\sum_{i'\in I_{n,u'}}Y_{n,i}}}
 \leq \frac Kn\sum_{i=[nt_{u-1}]+1}^{[nt_u]}
 \sum_{i'=[nt_{u'-1}]+1}^{[nt_{u'}]}\alpha\pr{i'-i}.
\end{equation}
Doing the change of inded $j=i'-i$ in the inner sum, 
it follows that 
\begin{equation}
 \frac 1n
 \abs{\cov{\sum_{i\in I_{n,u}}Y_{n,i}}{\sum_{i'\in I_{n,u'}}Y_{n,i}}}
 \leq \frac Kn\sum_{i=[nt_{u-1}]+1}^{[nt_u]}
 \sum_{j=[nt_{u'-1}]-i+1}^{[nt_{u'}]-i}\alpha\pr{j}
\end{equation}
and since $[nt_u]\leq [nt_{u'-1}]+1$ for $n$ large enough, we 
get 
\begin{equation}
 \frac 1n
 \abs{\cov{\sum_{i\in I_{n,u}}Y_{n,i}}{\sum_{i'\in I_{n,u'}}Y_{n,i}}}
 \leq \frac Kn\sum_{i= 1}^{[nt_{u'-1}]}
 \sum_{j=[nt_{u'-1}]-i+1}^{+\infty}\alpha\pr{j};
\end{equation}
doing the change $k= [nt_{u'-1}]-i $ gives 
\begin{equation}
 \frac 1n
 \abs{\cov{\sum_{i\in I_{n,u}}Y_{n,i}}{\sum_{i'\in I_{n,u'}}Y_{n,i}}}
 \leq \frac Kn\sum_{k= 0}^{[nt_{u'-1}]-1}
 \sum_{j=k+1}^{+\infty}\alpha\pr{j}
\end{equation}
which goes to zero in view of assumption~\ref{asp:alpha_mixing}. 
Therefore,
\begin{equation}
 \lim_{n\to+\infty}\frac 1n\E{\pr{\sum_{i=1}^nY_{n,i} }^2}
 =\sum_{u=1}^{d+1}\lim_{n\to+\infty}\frac 1n
 \E{\pr{\sum_{i\in I_{n,u}}Y_{n,i  }}^2}.
\end{equation}
By stationarity, 
\begin{equation}
 \E{\pr{\sum_{i\in I_{n,u}}Y_{n,i  }}^2}
 =\E{\pr{\sum_{i=1}^{[nt_u]-[nt_{u-1}]}Y_{n,i  }}^2}
\end{equation}
and defining for $1\leq u\leq d+1$ the random variable
\begin{multline}
 Z_{ i,u}:= \sum_{k,\ell=1}^d a_{k,\ell} 
 \pr{1- t_k} \mathbf 1\ens{u\leq k}
\pr{h_{1,s_\ell}\pr{X_i}-\E{h_{1,s_\ell}\pr{X_i}}} \\
+\sum_{k,\ell=1}^d a_{k,\ell}
 t_k \mathbf 1\ens{u\geq k+1}
 \pr{h_{2,s_\ell}\pr{X_i}-\E{h_{2,s_\ell}\pr{X_i}}},
\end{multline}
we have $\abs{Y_{n,i}-Z_{ i,u}}\leq K/n$ for a 
constant $K$ independent of $n$ and $i$ hence 
\begin{equation}
 \lim_{n\to+\infty}\frac 1n\E{\pr{\sum_{i=1}^nY_{n,i} }^2}
 =\sum_{u=1}^{d+1}\lim_{n\to+\infty}\frac 1n
 \E{\pr{\sum_{i=1}^{[nt_u]-[nt_{u-1}]}Z_{i,u  }}^2}.
\end{equation}
By expanding the square and using stationarity, it follows that 
\begin{equation}
 \lim_{n\to+\infty}\frac 1n\E{\pr{\sum_{i=1}^nY_{n,i} }^2}=
 \sigma^2
\end{equation}
where  
 
\begin{equation}
 \sigma^2=  \sum_{u=1}^{d+1}\pr{t_u-t_{u-1}}
 \sum_{i\in\Z}\cov{Z_{0,u}  }{Z_{i,u}}.
\end{equation}
This is the variance of $\sum_{k,\ell=1}^da_{k,\ell}
N_{k,\ell}$, where $\pr{N_{k,\ell}}_{k,\ell=1}^d$ is a centered 
Gaussian vector having  covariance 
\begin{equation}
 \operatorname{Cov}\pr{N_{k,\ell},N_{k',\ell'}}
 = \sum_{u=1}^{d+1}\pr{t_u-t_{u-1}}
 \sum_{i\in\Z}\E{Z_{k,\ell,0}^{\pr{u}},Z_{k',\ell',i}^{\pr{u}}    },
\end{equation}
where 
\begin{multline}
Z_{k,\ell,i}^{\pr{u}}= \pr{1-t_k}\mathbf 1\ens{u\leq k}
\pr{h_{1,s_\ell}\pr{X_i}-\E{h_{1,s_\ell}\pr{X_i}}}\\+ 
 t_k\mathbf 1\ens{u\geq k+1}\pr{h_{2,s_\ell}\pr{X_i}-\E{h_{2,s_\ell}\pr{X_i}}}.
\end{multline}
We can also check by spliting the sum over $u$ that for $k\leq 
k'$, 
\begin{equation}
 \operatorname{Cov}\pr{N_{k,\ell},N_{k',\ell'}}
 = \operatorname{Cov}\pr{W\pr{s_\ell ,t_k},
 W\pr{s_{\ell'},t_{k'}}}.
\end{equation}

Now, it remains to check the third condition of 
Theorem~\ref{thm:TLC_mixing_array}. We take 
$a_k:=\alpha\pr{k}$ and since $Y_{n,i}$ is a 
function of $X_i$, the inequality $\alpha_n\pr{k}\leq 
\alpha\pr{k}$ holds.

This ends the proof of Proposition~\ref{prop:conv_fidi}.

\end{proof}

\subsubsection{Convergence of the linear part}

\begin{Proposition} \label{prop:linear_part}
For all $R>0$, the sequence $\pr{W_n\pr{s,t},s\in [-R,R],t\in [0,1]}_{n\geq 1}$ converges in distribution in $D\pr{[-R,R] 
\times [0,1]}$ to $\pr{W\pr{s,t},s\in [-R,R],t\in [0,1]}$.
\end{Proposition}

In order to prove Proposition~\ref{prop:linear_part}, we will
use the following convergence criterion in 
$D\pr{[0,1]\times [0,1]}$, which is Corollary~1 in\cite{davydov:zitikis:2008}. 

\begin{Theorem}\label{thm:Davydov_Zitikis}
Let $\xi_n$, $n\geq 1$ be stochastic processes defined on $\left[0,1\right]^2$, taking 
values in $\mathbb R$, and whose paths are in the space $D\pr{[0,1]^2}$ 
almost surely. We make the following assumptions:
\begin{enumerate}
\item the finite-dimensional distributions of $\xi_n$ converges to the corresponding ones of a 
process $\xi$ having continuous paths;
\item\label{itm:diff_dec} the process $\xi_n$ can be written as the difference to two coordinate-wise 
non decreasing processes $\xi_n^\circ$ and $\xi_n^*$.
\item\label{itm:cond_1DZ} the exists constants $\gamma\geq \beta>2$, $c\in\pr{0,\infty}$ such that for all 
$n\geq 1$, $\E{\abs{\xi_n\pr{0,0}}^\gamma}\leq c$ and 
\begin{equation}\label{eq:tightness_moment_condition}
\E{\abs{\xi_n\pr{s,t}-\xi_n\pr{s',t'} }^\gamma}\leq c\norm{(s,t)-(s',t')}_\infty^\beta
\mbox{ whenever }\norm{(s,t)-(s',t')}_\infty\geq n^{-1}.
\end{equation}
\item\label{itm:cond_2DZ} the following convergence in probability holds:
\begin{equation}\label{eq:tightness_maximum_condition}
\max_{ 1\leq j_1,j_2\leq n }    \abs{
\xi_n^*\pr{ \frac{j_1}n,\frac{j_2}n   } -\xi_n^*
 \pr{ \frac{j_1-1}n,\frac{j_2}n   } 
}+\abs{
\xi_n^*\pr{ \frac{j_1}n,\frac{j_2}n   } -\xi_n^*
 \pr{ \frac{j_1}n,\frac{j_2-1}n   } 
}\to 0,
\end{equation}

\end{enumerate}
Then $\pr{\xi_n}_{n\geq 1}$ converges weakly to $\xi$ in 
$D\pr{[0,1]^2}$.
\end{Theorem}

In order to prove Proposition~\ref{prop:conv_fidi}, we will 
use Theorem~\ref{thm:Davydov_Zitikis} in the following setting: 
\begin{equation}
 \xi_n\pr{s,t}:=W_n\pr{-R+2Rs,t}, s,t\in [0,1]. 
\end{equation}
The convergence of the finite dimensional distributions to those of 
$\pr{W\pr{-R+2Rs,t},s,t\in [0,1]}$ is guaranted by Proposition~\ref{prop:conv_fidi}. 
The covariance function of this process is Lipschitz continuous hence this process has continuous paths. 

We will now express $\xi_n$ as the difference of two coordinatewise non-decreasing process. 
To ease the notations, we will write $F_{1,s}\pr{u}:=\PP\ens{g\pr{u,X_1}\leq -R+2Rs}$ and $F_{2,s}\pr{u}:=\PP\ens{g\pr{X_1,v}\leq -R+2Rs}$ and 
$m_{i,s}:=\E{F_{i,s}\pr{X_0}},i\in\ens{1,2} $. These four functions are non-negative and non-decreasing in $s$. The following equalities take place:
\begin{multline*}
\xi_n\pr{s,t}= \frac{n-[nt]}{n^{3/2}}\sum_{i=1}^{[nt]}\pr{F_{1,s}\pr{X_i}-m_{1,s}}+\frac{[nt]}{n^{3/2}}\sum_{j=[nt]+1}^n\pr{F_{2,s}\pr{X_j}-m_{2,s}}\\
= \frac{n-[nt]}{n^{3/2}}\sum_{i=1}^{[nt]} F_{1,s}\pr{X_i}+\frac{[nt]}{n^{3/2}}\sum_{j=[nt]+1}^n F_{2,s}\pr{X_j}
-m_{1,s}\frac{n-[nt]}{n^{3/2}}[nt]-m_{2,s}\frac{[nt]}{n^{3/2}}\pr{n-[nt]},
\end{multline*}
and continuing the decomposition, we obtain 
\begin{multline}
\xi_n\pr{s,t}=\frac 1{n^{1/2}}\sum_{i=1}^{[nt]}F_{1,s}\pr{X_i}-\frac{[nt]}{n^{3/2}}\sum_{i=1}^{[nt]}F_{1,s}\pr{X_i}+
\frac{[nt]}{n^{3/2}}\sum_{j=1}^n F_{2,s}\pr{X_j}-\frac{[nt]}{n^{3/2}}\sum_{j=1}^{[nt]} F_{2,s}\pr{X_j}\\
-m_{1,s}\frac{[nt]}{n^{1/2}}-m_{2,s}\frac{[nt]}{n^{1/2}} + m_{1,s} \frac{[nt]^2}{n^{3/2}}+m_{2,s} \frac{[nt]^2}{n^{3/2}}.
\end{multline}
Therefore, defining 
\begin{equation}
\xi_n^o\pr{s,t}:=\frac 1{n^{1/2}}\sum_{i=1}^{[nt]}F_{1,s}\pr{X_i}+\frac{[nt]}{n^{3/2}}\sum_{j=1}^n F_{2,s}\pr{X_j}
+ m_{1,s} \frac{[nt]^2}{n^{3/2}}+m_{2,s} \frac{[nt]^2}{n^{3/2}} \mbox{ and }
\end{equation}
\begin{equation}\label{eq:definition_de_xi_etoile}
\xi_n^*\pr{s,t}:= \frac{[nt]}{n^{3/2}}\sum_{i=1}^{[nt]}F_{1,s}\pr{X_i}+\frac{[nt]}{n^{3/2}}\sum_{j=1}^{[nt]} F_{2,s}\pr{X_j}+m_{1,s}\frac{[nt]}{n^{1/2}}+m_{2,s}\frac{[nt]}{n^{1/2}},
\end{equation}
the processes $\xi_n^o$ and $\xi_n^*$ are both coordinatewise non-decreasing and 
$\xi_n=\xi_n^o-\xi_n^*$. 

Now, let $p$ be such that Assumption~\ref{asp:alpha_mixing} 
is satisfied and let 
$\gamma:=2p$ and $\beta:=p$ (since $p>2$, the condition 
$\beta>2$ is fulfiled). In view of Theorem~\ref{thm:Davydov_Zitikis}, the proof of Proposition~\ref{prop:linear_part} 
will be complete once we show the following: 
\begin{equation}\label{eq:tightness_moment_condition_our_case_inc_en_t}
\sup_{n\geq 1}\sup_{s\in [0,1]}\sup_{ \substack{t,t'\in [0,1]\\ \abs{t'-t}\geq n^{-1} }}
\E{\abs{\xi_n\pr{s,t}-\xi_n\pr{s,t'}   }^{2p}}\leq  
C\abs{t-t'}^{p};
\end{equation}
\begin{equation}\label{eq:tightness_moment_condition_our_case_inc_en_s}
\sup_{n\geq 1}\sup_{ \substack{s,s'\in [0,1]\\ \abs{s'-s}\geq n^{-1} }}\sup_{t\in [0,1]}
\E{\abs{\xi_n\pr{s,t}-\xi_n\pr{s',t}   }^{2p }}\leq 
C\abs{s-s'}^{p};
\end{equation}
\begin{equation}\label{eq:tightness_maximum_condition_our_case_inc_en_s}
\max_{ 0\leq j_1,j_2\leq n  }    \abs{
\xi_n^*\pr{ \frac{j_1+1}n,\frac{j_2}n    } -\xi_n^*\pr{\frac{j_1}n,\frac{j_2}n    }
}\to 0\mbox{ in probability};
\end{equation}
\begin{equation}\label{eq:tightness_maximum_condition_our_case_inc_en_t}
\max_{ 0\leq j_1,j_2\leq n  }    \abs{
\xi_n^*\pr{ \frac{j_1}n,\frac{j_2+1}n    } -\xi_n^*\pr{\frac{j_1}n,\frac{j_2}n    }
}\to 0\mbox{ in probability}.
\end{equation}

In order to prove \eqref{eq:tightness_moment_condition_our_case_inc_en_t} and 
\eqref{eq:tightness_moment_condition_our_case_inc_en_s}, we need to control the 
differences in $s$ and $t$.

Let us show \eqref{eq:tightness_moment_condition_our_case_inc_en_t}. We first control 
$\abs{\xi_n\pr{s,t}-\xi_n\pr{s,t'}   }$ for $t'-t\geq 1/n$. By definition of $\xi_n$, 
\begin{multline}
\abs{\xi_n\pr{s,t}-\xi_n\pr{s,t'}   }\leq 
\abs{\frac{n-[nt']}{n^{3/2}}\sum_{i=1}^{[nt']}\pr{F_{1,s}\pr{X_i}-m_{1,s}}  -
\frac{n-[nt]}{n^{3/2}}\sum_{i=1}^{[nt]}\pr{F_{1,s}\pr{X_i}-m_{1,s}}   }\\
+ \abs{\frac{[nt']}{n^{3/2}}\sum_{i=[nt']+1}^{n}\pr{F_{2,s}\pr{X_i}-m_{2,s}}  -
\frac{[nt]}{n^{3/2}}\sum_{i=[nt]+1}^{n}\pr{F_{2,s}\pr{X_i}-m_{2,s}}   }.
\end{multline}
Decomposing the sum $\sum_{i=1}^{[nt']}$ as $\sum_{i=1}^{[nt]}+
\sum_{i=[nt]+1}^{[nt']}$ and the sum $\sum_{i=[nt]+1}^{n}$ as 
$\sum_{i=[nt]+1}^{[nt']}+\sum_{i=[nt']+1}^{n}$, we derive that 
\begin{multline}
\abs{\xi_n\pr{s,t}-\xi_n\pr{s,t'}   }\leq 
\frac{[nt']-[nt]}{n^{3/2}}    \abs{  
\sum_{i=1}^{[nt]}\pr{F_{1,s}\pr{X_i}-m_{1,s}}   }+
\frac{n-[nt']}{n^{3/2}}\abs{\sum_{i=[nt]+1}^{[nt']}\pr{F_{1,s}\pr{X_i}-m_{1,s}}}\\
+ \frac{[nt']-[nt]}{n^{3/2}}\abs{\sum_{i=[nt']+1}^{n}\pr{F_{2,s}\pr{X_i}-m_{2,s}} } 
\frac{[nt]}{n^{3/2}}\abs{\sum_{i=[nt]+1}^{[nt'] }\pr{F_{2,s}\pr{X_i}-m_{2,s}}   }.
\end{multline}
Taking into account that $t'-t\geq 1/n$, it follows that 
\begin{equation}\label{eq:control_diff_parties_entieres}
[nt']-[nt]\leq nt'-nt+1\leq 2n\pr{t'-t}
\end{equation}
hence 
\begin{multline}\label{eq:differences_en_t}
\abs{\xi_n\pr{s,t}-\xi_n\pr{s,t'}   }\leq 
2\frac{t'-t}{n^{1/2}}    \abs{  
\sum_{i=1}^{[nt]}\pr{F_{1,s}\pr{X_i}-m_{1,s}}   }+
\frac{1}{n^{1/2}}\abs{\sum_{i=[nt]+1}^{[nt']}\pr{F_{1,s}\pr{X_i}-m_{1,s}}}\\
+ 2\frac{t'-t}{n^{1/2}}\abs{\sum_{i=[nt']+1}^{n}\pr{F_{2,s}\pr{X_i}-m_{2,s}} } 
+\frac{1}{n^{1/2}}\abs{\sum_{i=[nt]+1}^{[nt'] }\pr{F_{2,s}\pr{X_i}-m_{2,s}}   }.
\end{multline}

By stationarity, 
\begin{multline}\label{eq:differences_en_t_puis_p}
 \E{\abs{\xi_n\pr{s,t}-\xi_n\pr{s,t'}   }^{2p}  }
 \leq C_p\pr{\frac{t'-t}{n^{1/2}}}^{2p}
 \E{\abs{  
\sum_{i=1}^{[nt]}\pr{F_{1,s}\pr{X_i}-m_{1,s}}   }^{2p}}\\
+C_pn^{-p}\E{\abs{\sum_{i=1}^{[nt']-[nt]}\pr{F_{1,s}\pr{X_i}-m_{1,s}}}^{2p}}
+C_p \pr{\frac{t'-t}{n^{1/2}}}^{2p}
\E{\abs{\sum_{i=1}^{n-[nt']}\pr{F_{2,s}\pr{X_i}-m_{2,s}} }^{2p} }\\
+C_p n^{-p}\E{\abs{\sum_{i=1}^{[nt']-[nt] }\pr{F_{2,s}\pr{X_i}-m_{2,s}}  }
}^{2p}
\end{multline}

By application of Proposition~\ref{prop:moment_p_mixing}
to each of the four terms of the right 
hand side of \eqref{eq:differences_en_t_puis_p} and using boundedness of $ F_{1,s}\pr{X_i}-m_{1,s} $ and 
$ F_{2,s}\pr{X_i}-m_{2,s} $ by $2$, we derive that for all $s,t,t'\in [0,1]$ such that 
$t'-t\geq 1/n$, 
\begin{multline}
\E{\abs{\xi_n\pr{s,t}-\xi_n\pr{s,t'}   }^{2p}}
\leq C\pr{p,\pr{\alpha\pr{k}}_{k\geq 1}}
\pr{t'-t}^{2p}n^{p}\\+
 C\pr{p,\pr{\alpha\pr{k}}_{k\geq 1}} 
 \pr{\frac{[nt']-[nt]}{n}  }^{p}+
 C\pr{p,\pr{\alpha\pr{k}}_{k\geq 1}} \pr{t'-t}^{2p}.
\end{multline}
and using again \eqref{eq:control_diff_parties_entieres}, we derive 
\eqref{eq:tightness_moment_condition_our_case_inc_en_t}.

Let us show \eqref{eq:tightness_moment_condition_our_case_inc_en_s}. This 
time, we have to control the increments in the first variable. For $s,s',t\in [0,1]$ such that $s'-s\geq 1/n$, 
\begin{multline}
\abs{\xi_n\pr{s',t}-\xi_n\pr{s,t}   }\leq 
  \frac{n-[nt]}{n^{3/2}}
  \abs{\sum_{i=1}^{[nt]}\pr{F_{1,s}\pr{X_i}- F_{1,s'}\pr{X_i}+m_{1,s'}-m_{1,s}}}\\ +
  \frac{[nt]}{n^{3/2}}
  \abs{\sum_{i=[nt]+1}^{n}\pr{F_{2,s}\pr{X_i}-F_{2,s'}\pr{X_i}+m_{2,s'}-m_{2,s}}}.
\end{multline}
Bounding $n-[nt]$  and $[nt]$ by $n$, 
we derive by stationarity that 
\begin{multline}
\E{\abs{\xi_n\pr{s',t}-\xi_n\pr{s,t}   }^{2p}}
\leq C_p n^{-p}
\E{\abs{\sum_{i=1}^{[nt]}\pr{F_{1,s}\pr{X_i}- F_{1,s'}\pr{X_i}+m_{s'}-m_{1,s}}}^{2p}}\\
+ C_p n^{-p}\E{
  \abs{\sum_{i=1}^{n-[nt]}\pr{F_{2,s}\pr{X_i}-F_{2,s'}\pr{X_i}+m_{2,s'}-m_{2,s}}}^{2p}
}
\end{multline}

and taking into account the 
inequalities $\abs{F_{1,s}\pr{x}-F_{2,s'}\pr{x}}\leq 2RM\abs{s-s'}$ and 
$\abs{F_{2,s}\pr{x}-F_{2,s'}\pr{x}}\leq 2RM\abs{s-s'}$, we derive that 
\begin{equation}
\E{\abs{\xi_n\pr{s',t}-\xi_n\pr{s,t}   }^{2p }}
\leq   
C\pr{s'-s}^{p},
\end{equation}
showing \eqref{eq:tightness_moment_condition_our_case_inc_en_s}.

Now we show \eqref{eq:tightness_maximum_condition_our_case_inc_en_s}. Going back to the 
expression of $\xi_n^*$ given by \eqref{eq:definition_de_xi_etoile}, we derive that 
\begin{multline}
\abs{\xi_n^*\pr{s,t}-\xi_n^*\pr{s',t}}\leq 
\frac 1{n^{3/2}}\sum_{i=1}^n\pr{\abs{F_{1,s}\pr{X_i}-F_{1,s'}\pr{X_i}} +
\abs{F_{2,s}\pr{X_i}-F_{2,s'}\pr{X_i}}   }\\+ \sqrt n\abs{m_{1,s}-m_{1,s'}}
+\sqrt n\abs{m_{2,s}-m_{2,s'}}
 \leq 8R\sqrt n\abs{s'-s}
\end{multline}
hence 
\begin{equation}
\max_{ 0\leq j_1,j_2\leq n  }    \abs{
\xi_n^*\pr{ \frac{j_1+1}n,\frac{j_2}n    } -\xi_n^*\pr{\frac{j_1}n,\frac{j_2}n    }
}\leq \frac{8R}{\sqrt n},
\end{equation}
giving \eqref{eq:tightness_maximum_condition_our_case_inc_en_s}.

Finally, we show \eqref{eq:tightness_maximum_condition_our_case_inc_en_t}.
For $0 \leq t\leq t'\leq 1$ and $s\in [0,1]$, denoting $c_s:=F_{1,s}+F_{2,s}$ and $m_s=m_{1,s}+m_{2,s}$, the following inequalities hold:
\begin{align}
\abs{\xi_n^*\pr{s,t}-\xi_n^*\pr{s,t'}}&\leq 
\abs{\frac{[nt']}{n^{3/2}}\sum_{i=1}^{[nt']}c_s\pr{X_i}-   \frac{[nt]}{n^{3/2}}\sum_{i=1}^{[nt]}c_s\pr{X_i}}
+ m_s\frac{[nt']-[nt]}{n^{1/2}}\\
&\leq \frac{[nt']-[nt] }{n^{3/2}}\sum_{i=1}^{[nt]}c_s\pr{X_i}+\frac{[nt']}{n^{3/2}}\sum_{i=[nt]+1}^{[nt']}c_s\pr{X_i}+2m_s\frac{[nt']-[nt]}{n^{1/2}}\\
&\leq 4\frac 1{n^{3/2}}\pr{ [nt']-[nt]   }\pr{ [nt]+[nt']   }+8\frac{[nt']-[nt]}{n^{1/2}}\\
&\leq \frac{8}{n^{3/2}}\pr{ n\pr{t'-t}+1  }+8\frac{n\pr{t'-t}+1}{n^{1/2}}\\
&\leq 16 n^{1/2}\pr{t'-t}+16n^{-1/2}.
\end{align}
As a consequence, we derive that 
\begin{equation}\label{eq:final_step_linear_part}
\max_{ 0\leq j_1,j_2\leq n  }    \abs{
\xi_n^*\pr{ \frac{j_1}n,\frac{j_2+1}n    } -\xi_n^*\pr{\frac{j_1}n,\frac{j_2}n    }
}\leq 32n^{-1/2},
\end{equation}
and \eqref{eq:tightness_maximum_condition_our_case_inc_en_t} follows. This ends the proof of 
Proposition~\ref{prop:linear_part}.
\subsubsection{Negligibility of the degenerated part}

\begin{Proposition}\label{prop:neg_terme_deg}
Let  $R_n$ be defined by 
\eqref{eq:definition_de_Rn}.  Under the conditions of Theorem~\ref{thm:convergence_two_sample_emp_Ustat}, the 
following convergence holds:
\begin{equation}
\sup_{-R\leq s\leq R}\sup_{0\leq t\leq 1}\abs{ R_n\pr{s,t}}\to 0\mbox{ in probability}.
\end{equation}
\end{Proposition}

 As a first step, we can look to
the control of the supremum on $t$ for a fixed $s$. Then 
\begin{equation}
\sup_{0\leq t\leq 1}\abs{ R_n\pr{s,t}}
=\frac 1{n^{3/2}}\max_{1\leq \ell \leq n-1}\abs{\sum_{i=1}^\ell\sum_{j=\ell+1}^n\pr{h_{3,s}\pr{X_i,X_j}-\E{h_{3,s}\pr{X_i,X_j}}}}.
\end{equation}
The problem 
is that the index $\ell$ over which we take the maximum appear in both sums and we cannot directly apply maximal inequality 
in Lemma~\ref{lem:moment_ineg_Ustats}. 

Here is a way to overcome this issue. Denote $h_{i,j}:= 
 h_{3,s}\pr{X_i,X_j}-\E{h_{3,s}\pr{X_i,X_j}} $ and for a fixed $n$, 
$S_\ell:=\sum_{i=1}^\ell\sum_{j=\ell+1}^nh_{i,j} $, $1\leq \ell\leq n-1$ and $S_0=0$.
Then for $2\leq \ell\leq n$, 
\begin{align}
S_{\ell }-S_{\ell-1} &= \sum_{i=1}^{\ell }\sum_{j=\ell+1}^nh_{i,j}-\sum_{i=1}^{\ell-1}\sum_{j=\ell }^nh_{i,j}\\
&= \sum_{i=1}^{\ell -1}\sum_{j=\ell+1}^nh_{i,j}+  \sum_{j=\ell+1}^nh_{\ell,j} -\sum_{i=1}^{\ell-1}\sum_{j=\ell+1}^nh_{i,j}
-\sum_{i=1}^{\ell-1}h_{i,\ell}\\
&=    \sum_{j=\ell+1}^nh_{\ell ,j} 
-\sum_{i=1}^{\ell-1} h_{i,\ell }
\end{align}
hence 
\begin{equation}\label{eq:dec_part_deg_mart_rev_mart}
S_k=\sum_{i=1}^k\sum_{j=k+1}^nh_{i,j}=\sum_{\ell=1}^{k }\pr{S_{\ell }-S_{\ell-1}}=
\sum_{\ell=1}^{k }\sum_{j=\ell+1}^nh_{\ell ,j}-\sum_{1\leq i<\ell\leq k}h_{i,\ell}.
\end{equation}

Therefore, the maximum over $k$ can be treated by using 
the maximum of degenerated  $U$-statistic for the 
term $\sum_{1\leq i<\ell\leq k}h_{i,\ell}$. The other one can 
be reduced to a similar contribution by using this 
time the data and $\pr{X_n,\dots,X_1}$ instead of $\pr{X_1,\dots,X_n}$.

Using Lemma~\ref{lem:moment_ineg_Ustats}, we derive that 
for each fixed $s$, 
 \begin{equation}\label{eq:controle_proba_Rn_s_fixed}
  \PP\ens{\sup_{0\leq t\leq 1}\abs{ R_n\pr{s,t}}>\eps}\leq 
   C\eps^{-2}\frac{\log n}n, 
 \end{equation}
 where the constant $C$ depends only on 
 $\pr{\beta\pr{k}}_{k\geq 1}$.

We divide the interval $[-R,R]$ in 
intervals of length $\delta_n$, where $\delta_n$ with be 
specified later. Let
$a_k:= -R+2Rk\delta_n$ and the interval
\begin{equation}
I_k:= [a_{k-1},a_{k}], 1\leq k\leq 
\left[1/\delta_n   \right]+1=:B_n
\end{equation}
(here for simplicity, we ommit the dependence in $n$ for $a_k$ and $I_k$).

Then 
\begin{equation}\label{eq:first:bound_max_between_-AnAn}
\sup_{-R\leq s\leq R}\sup_{0\leq t\leq 1}\abs{ R_n\pr{s,t}}
\leq \max_{1\leq k\leq B_n}
\sup_{s\in I_k}\sup_{0\leq t\leq 1}\abs{ R_n\pr{s,t}}.
\end{equation}
In order to handle the supremum on $I_k$, we need the following lemma:
\begin{Lemma}\label{lem:control_diff_de_deux_fcts_croissantes}
Let $a$ and $b$ be two real numbers such that $a<b$ and let 
$f\colon \R\to \R$ be a function which can be expressed as a difference 
of two non-decreasing functions $f_1$ and $f_2$. Then 
\begin{equation}
\sup_{s\in [a,b]}\abs{f\pr{s}}\leq \abs{f\pr{a}}+\abs{f\pr{b}}+
 \abs{f_2\pr{b}-f_2\pr{a}} .
\end{equation}
\end{Lemma}
\begin{proof}
Let $s\in [a,b]$. Then by non-decreasingness of $f_1$ and $f_2$, 
\begin{equation*}
f\pr{s}=f_1\pr{s}-f_2\pr{s}\leq f_1\pr{b}-f_2\pr{a}=f\pr{b}+f_2\pr{b}-f_2\pr{a}
\leq  \abs{f\pr{a}}+\abs{f\pr{b}}+\abs{f_2\pr{b}-f_2\pr{a}} .
\end{equation*}
Moreover, 
\begin{equation*}
f\pr{s}=f_1\pr{s}-f_2\pr{s}\geq f_1\pr{a}-f_2\pr{b}=f\pr{a}+f_2\pr{a}-f_2\pr{b}
\geq -\abs{f\pr{a}}-\abs{f_2\pr{b}-  f_2\pr{a} },
\end{equation*}
which allows to conclude. 
\end{proof}
In the expression $\sup_{s\in I_k}\sup_{0\leq t\leq 1}\abs{ R_n\pr{s,t}}$, the supremum 
over $t$ is actually a maximum; for the supremum over $s$, we will apply 
Lemma~\ref{lem:control_diff_de_deux_fcts_croissantes} in the following setting: 
for a fixed $t\in [0,1]$, 
\begin{equation}
f_1\pr{s}= \sum_{i=1}^{[nt]}\sum_{j=[nt]+1}^n\pr{ \mathbf 1\ens{g\pr{X_i,X_j}\leq s}+\E{h_{1,s}\pr{X_i}}+\E{h_{1,s}\pr{X_i}}};
\end{equation}
\begin{equation}
f_2\pr{s}= \sum_{i=1}^{[nt]}\sum_{j=[nt]+1}^n\pr{   
\widetilde{h_{1,s}}\pr{X_i}+\widetilde{h_{2,s}}\pr{X_j}
+\PP\ens{g\pr{X_i,X_j}\leq s}
},
\end{equation}
where $\widetilde{h_{1,s}}\pr{u}=\PP\ens{g\pr{u,X_2}\leq s}$ and 
$\widetilde{h_{2,s}}\pr{v}=\PP\ens{g\pr{X_1,v}\leq s}$.
In view of \eqref{eq:first:bound_max_between_-AnAn}, we derive that 
\begin{equation}\label{eq:bound_max_entre_pm_An}
\sup_{-R\leq s\leq R}\sup_{0\leq t\leq 1}\abs{ R_n\pr{s,t}}
\leq Z_n+c_n
\end{equation}
where 
\begin{multline}
 Z_n:=2\max_{1\leq k\leq B_n}
 \sup_{0\leq t\leq 1}\abs{ R_n\pr{a_k,t}}\\
 +\frac 1{n^{3/2}}\sup_{0\leq t\leq 1}\max_{1\leq k\leq B_n} 
 \abs{\sum_{i=1}^{[nt]}\sum_{j=[nt]+1}^n\pr{   
\widetilde{h_{1,a_k}}\pr{X_i}-\widetilde{h_{1,a_{k-1}}}\pr{X_i}+
\widetilde{h_{2,a_k}}\pr{X_j}-\widetilde{h_{2,a_{k-1}}}\pr{X_j}
}}
\end{multline}
and 
\begin{equation}\label{eq:definition_de_cn}
 c_n:= \frac 1{n^{3/2}}\sup_{0\leq t\leq 1}\max_{1\leq k\leq B_n} 
\sum_{i=1}^{[nt]}\sum_{j=[nt]+1}^n\ens{\PP\ens{g\pr{X_i,X_j}\leq a_k}
-\PP\ens{g\pr{X_i,X_j}\leq a_{k-1}}},
\end{equation}
We first show that $\pr{Z_n}_{n\geq 1}$ goes to zero in probability.
After having rearranged the second term of the right hand side of 
\eqref{eq:bound_max_entre_pm_An}, we end up with the estimate 
\begin{multline}\label{eq:first_estimate_conv_proba_max_interval}
\PP\ens{Z_n>4\eps}
\leq \PP\ens{ \max_{1\leq k\leq B_n}
 \sup_{0\leq t\leq 1}\abs{ R_n\pr{a_k,t}}>\eps   }\\ 
 +\PP\ens{\frac 1{\sqrt n}\max_{1\leq k\leq B_n}  
 \sum_{i=1}^n\pr{
\widetilde{h_{1,a_k}}\pr{X_i}-\widetilde{h_{1,a_{k-1}}}\pr{X_i}}>\eps}\\
+\PP\ens{\frac 1{\sqrt n}\max_{1\leq k\leq B_n}  \sum_{j=1}^n
\pr{\widetilde{h_{2,a_k}}\pr{X_j}-\widetilde{h_{2,a_{k-1}}}\pr{X_j}}>\eps
}.
\end{multline}
Let us estimate the first term of the right hand side of \eqref{eq:first_estimate_conv_proba_max_interval}. 
The use of \eqref{eq:controle_proba_Rn_s_fixed} and a union bound yields 
\begin{equation}
\PP\ens{ \max_{1\leq k\leq B_n}
 \sup_{0\leq t\leq 1}\abs{ R_n\pr{a_k,t}}>\eps   }\leq CB_n\eps^{-2}\frac{\log n}n.
\end{equation}

Now, using the assumptions \ref{asp:bounded_density_coord_1} 
and assumption \ref{asp:bounded_density_coord_2}, we 
derive that 
 
\begin{equation}
 \max_{1\leq k\leq B_n}  
 \sum_{i=1}^n\pr{
\widetilde{h_{q,a_k}}\pr{X_i}-\widetilde{h_{q,a_{k-1}}}\pr{X_i}}\leq  
n \max_{1\leq k\leq B_n} M_q\pr{a_k-a_{k-1}}\leq M_q n\delta_n ,q\in\ens{1,2}
\end{equation} 
hence the condition 
\begin{equation}
\lim_{n\to +\infty} \sqrt{n}\delta_n=0
\end{equation}
guarantees the convergence in probability of the last two terms of \eqref{eq:first_estimate_conv_proba_max_interval}.
 
We thus need 
\begin{equation}
 \lim_{n\to +\infty} \sqrt{n}\delta_n=0; 
 \lim_{n\to +\infty}\frac{\log n}{n\delta_n}=0,
\end{equation}
which can be done by choosing 
$\delta_n=n^{-3/4}$. 

It remains to show the convergence to zero of the sequence $\pr{c_n}_{n\geq 1}$ defined by 
\eqref{eq:definition_de_cn}. To this aim, we estimate for all $i<j$ 
the probability $p_{i,j}:=\PP\ens{g\pr{X_i,X_j}\in (a_{k-1},a_k]}$ 
by using coupling. We can find a random variable $X'_j$ independent 
of $X_i$ and having the same distribution as $X_j$ such that 
$\PP\ens{X_j\neq X'_j}\leq \beta\pr{j-i}$. Therefore, 
$p_{i,j}\leq \beta\pr{j-i}+\PP\ens{g\pr{X_i,X'_j}\in (a_{k-1},a_k]}$ and
$c_n\leq c'_n+c''_n$, where 
\begin{equation}
 c'_n=\frac 1{n^{3/2}}\max_{1\leq \ell\leq n}
\sum_{i=1}^{\ell}\sum_{j=\ell+1}^n \beta\pr{j-i}
\end{equation}
\begin{equation}
 c''_n=\frac 1{n^{3/2}}\max_{1\leq \ell\leq n}\max_{1\leq k\leq B_n} 
\sum_{i=1}^{\ell}\sum_{j=\ell+1}^n  \PP\ens{g\pr{X_i,X'_j}\in (a_{k-1},a_k]}.
\end{equation}
Let us show that $c'_n\to 0$. For a fixed $\ell\in\ens{1,\dots,n}$, 
\begin{equation}
 \sum_{i=1}^{\ell}\sum_{j=\ell+1}^n \beta\pr{j-i}
 \leq \sum_{i=1}^{\ell}\sum_{k\geq \ell-i}\beta\pr{k}
 \leq \ell \sum_{k\geq 0}\beta\pr{k}\leq n\sum_{k\geq 0}\beta\pr{k}
\end{equation}
hence 
\begin{equation}
 c'_n\leq n^{-1/2}\sum_{k\geq 0}\beta\pr{k}
\end{equation}
and by assumption \ref{asp:beta_mixing}, we derive that 
$c'_n\to 0$. 

Let us show that $c''_n\to 0$. First, we notice that for 
all $i<j$, the vector $\pr{X_i,X'_j}$ has the same distribution as 
$\pr{X,Y}$, where $X$ and $Y$ are independent and have the same 
distribution as $X_1$. Consequently, 
\begin{equation}
c''_n\leq \sqrt{n} \max_{1\leq k\leq B_n} 
\PP\ens{g\pr{X,Y}\in (a_{k-1},a_k]} 
\end{equation}
By assumption \ref{asp:bounded_density_coord_1} and accounting 
$\delta_n=n^{-3/4}$, we derive that 
\begin{equation}
c''_n\leq M_1\sqrt{n} \max_{1\leq k\leq B_n} 
\pr{a_k-a_{k-1}}\leq M\sqrt{n}\delta_n\leq Mn^{1/2-3/4}\to 0.
\end{equation}

This 
ends the proof of Theorem~\ref{thm:convergence_two_sample_emp_Ustat}.

\subsection{Proof of Theorem~\ref{thm:convergence_two_sample_emp_Ustat_prob_replaced}}

 \subsubsection{Hoeffding's decomposition}
 
Using the functions $h_{1,s}$, $h_{2,s}$ and $h_{3,s}$ defined by 
\eqref{eq:def_noyau_lin} and \eqref{eq:def_noyau_deg}, we derive that 
\begin{equation}\label{eq:decomposition_de_e'_n}
e'_n\pr{s,t}=W'_n\pr{s,t}+R'_n\pr{s,t},
\end{equation}
where 
\begin{multline}\label{eq:definition_de_W'n}
W'_n\pr{s,t}= \frac 1{n^{3/2}}\pr{n-[nt]}\sum_{i=1}^{[nt]}
\pr{h_{1,s}\pr{X_i}-\E{h_{1,s}\pr{X_i}}}+
\frac 1{n^{3/2}}[nt]\sum_{j=[nt]+1}^n\pr{h_{2,s}\pr{X_j}-\E{h_{2,s}\pr{X_j}}}\\
-\frac{[nt]\pr{n-[nt]}}{n^{3/2}}\frac1{{{n}\choose{2}}}\pr{\sum_{i=1}^{n-1}\pr{n-i}\pr{h_{1,s}\pr{X_i}-\E{h_{1,s}\pr{X_i}}}+
\sum_{j=2}^n\pr{j-1}\pr{h_{2,s}\pr{X_j}-\E{h_{2,s}\pr{X_j}}}}
\end{multline}
 and 
 \begin{multline}\label{eq:definition_de_R'n}
 R'_n\pr{s,t}:= \frac{1}{n^{3/2}}\sum_{i=1}^{[nt]}\sum_{j=[nt]+1}^n 
  \pr{h_{3,s}\pr{X_i,X_j}-\E{h_{3,s}\pr{X_i,X_j}} } \\-
   \frac{[nt]\pr{n-[nt]}}{n^{3/2}}
  \frac1{{{n}\choose{2}}}\sum_{1\leq i<j\leq n} \pr{h_{3,s}\pr{X_i,X_j}-\E{h_{3,s}\pr{X_i,X_j}}  }.
 \end{multline}

 \subsubsection{Convergence of the finite dimensional distributions}
 
We treat the convergence of the finite dimensional distributions.  

\begin{Proposition}\label{prop:conv_fidi_prob_remp}
For all $d\geq 1$ and all $s_1<\dots<s_d$ and $0\leq t_1
<\dots<t_d\leq 1$, the vector $\pr{W'_n\pr{s_\ell,t_k}}_{k,\ell
=1}^d$ converges in distribution to $\pr{W'\pr{s_\ell,t_k}}_{k,\ell
=1}^d$, where $W'_n$ is defined by \eqref{eq:definition_de_W'n} and 
$W$ is like in Theorem~\ref{thm:convergence_two_sample_emp_Ustat_prob_replaced}.
\end{Proposition} 
 
Here again, we will prove the convergence of linear combinations, that is, for all $\pr{a_{k,\ell}}_{k,\ell=1}^d$, 
the convergence in distribution
\begin{equation}
\sum_{k,\ell=1}^da_{k,\ell}W'_n\pr{s_\ell,t_k}\to \sum_{k,\ell=1}^da_{k,\ell}W'\pr{s_\ell,t_k}
\end{equation} 
holds. To this aim, we will express $\sum_{k,\ell=1}^da_{k,\ell}W'_n\pr{s_\ell,t_k}$ as a linear combination 
of a sum of functions of $X_i$. Using the definition of $I_{n,u}$ given by \eqref{eq:def_of_Inu} for 
$1\leq u\leq d+1$, we derive that 
 \begin{equation}
 \sum_{k,\ell=1}^da_{k,\ell}W'_n\pr{s_\ell,t_k}=
 \frac{1}{\sqrt n}\sum_{i=1}^{n} A_{n,i},
 \end{equation}
 
where, for $i\in I_{n,u}$, 
\begin{multline}
A_{n,i}=\frac 1{n}\sum_{k,\ell=1}^da_{k,\ell}
\pr{n-[nt_k]}\pr{h_{1,s_\ell}\pr{X_i}-\E{h_{1,s_\ell}\pr{X_i}}   }\mathbf 1\ens{u\leq k}\\ -
\frac 1{n }\frac1{{{n}\choose{2}}}\pr{n-i}\sum_{k,\ell=1}^da_{k,\ell}[nt_k]\pr{n-[nt_k]}\pr{h_{1,s_\ell}\pr{X_i}-\E{h_{1,s_\ell}\pr{X_i}}   }\\+
\frac 1{n }\sum_{k,\ell=1}^da_{k,\ell}[nt_k]\pr{h_{2,s_\ell}\pr{X_i}
-\E{h_{2,s_\ell}\pr{X_i}}
}
\mathbf 1\ens{u> k}\\
-\frac 1{n }\frac1{{{n}\choose{2}}}\pr{i-1}\sum_{k,\ell=1}^da_{k,\ell}[nt_k]\pr{n-[nt_k]}\pr{h_{2,s_\ell}\pr{X_i}
-\E{h_{2,s_\ell}\pr{X_i}}
}.
\end{multline} 
 
In order to get rid of terms of the form $\ent{nt_k}$ and 
reduce the dependence in $n$, we will define 
for $i\in I_{n,u}$ the random variable
\begin{multline}
Y_{n,i}= \sum_{k,\ell=1}^da_{k,\ell}
\pr{1-t_k }\pr{h_{1,s_\ell}\pr{X_i}-\E{h_{1,s_\ell}\pr{X_i}}   }\mathbf 1\ens{u\leq k}\\ -
\frac 2{n } \pr{n-i}\sum_{k,\ell=1}^da_{k,\ell}t_k\pr{1- t_k }\pr{h_{1,s_\ell}\pr{X_i}-\E{h_{1,s_\ell}\pr{X_i}}   }+
 \sum_{k,\ell=1}^da_{k,\ell}t_k\pr{h_{2,s_\ell}\pr{X_i}
-\E{h_{2,s_\ell}\pr{X_i}}
}\mathbf 1\ens{u> k}\\
-\frac 2{n } i \sum_{k,\ell=1}^da_{k,\ell} t_k \pr{1- t_k }\pr{h_{2,s_\ell}\pr{X_i}
-\E{h_{2,s_\ell}\pr{X_i}}
}.
\end{multline}
Since $\abs{A_{n,i}-Y_{n,i}}\leq 1/n$, it suffices 
to show that $n^{-1/2}\sum_{i=1}^nY_{n,i}=:n^{-1/2}S_n
\to \sum_{k,\ell=1}^da_{k,\ell}W'\pr{s_\ell,t_k}$ in distribution.
Here again, will use Theorem~\ref{thm:TLC_mixing_array}.  
The first condition can be easily checked by bounding 
$\abs{h_{q,s_\ell}\pr{X_i}}$ by $2$ and the terms $
i/n$ by $1$. For the third condition, 
we also use $a_k=\alpha\pr{k}$, since each random variable 
$Y_{n_i}$ is a function of $X_i$.
It remains to compute the limit of the sequence 
$\pr{n^{-1}\operatorname{Var}\pr{S_n}}_{n\geq 1}$. 

By similar argument as those who gave \eqref{eq:cov_blocs_disjoints}, 
this reduces to compute for all $1\leq u\leq d+1$ the limit 
\begin{equation}
 \lim_{n\to +\infty}\frac 1n\E{
 \pr{\sum_{i\in I_{n,u}}Y_{n,i} }^2
 }.
\end{equation}

For convenience, we write 
\begin{multline}
Z_i:=  \sum_{k,\ell=1}^da_{k,\ell}
\pr{1-t_k }\pr{h_{1,s_\ell}\pr{X_i}-\E{h_{1,s_\ell}\pr{X_i}}   }\mathbf 1\ens{u\leq k} \\
+ \sum_{k,\ell=1}^da_{k,\ell}t_k\pr{h_{2,s_\ell}\pr{X_i}
-\E{h_{2,s_\ell}\pr{X_i}}
}\mathbf 1\ens{u> k}\\
-2 \sum_{k,\ell=1}^da_{k,\ell}t_k\pr{1- t_k }\pr{h_{1,s_\ell}\pr{X_i}-\E{h_{1,s_\ell}\pr{X_i}}   };
\end{multline}
\begin{equation}
 Z'_{n,i}= \frac 2{n } i \sum_{k,\ell=1}^da_{k,\ell} t_k \pr{1- t_k }\pr{\pr{h_{1,s_\ell}\pr{X_i}-\E{h_{1,s_\ell}\pr{X_i}}   }-\pr{h_{2,s_\ell}\pr{X_i}
-\E{h_{2,s_\ell}\pr{X_i}}
}},
\end{equation}
\begin{equation}
 Z''_i:= 2\sum_{k,\ell=1}^da_{k,\ell} t_k \pr{1- t_k }\pr{\pr{h_{1,s_\ell}\pr{X_i}-\E{h_{1,s_\ell}\pr{X_i}}   }-\pr{h_{2,s_\ell}\pr{X_i}
-\E{h_{2,s_\ell}\pr{X_i}}
}}.
\end{equation}

 Then $Y_{n,i}=Z_i+Z'_i$ and consequently
 \begin{equation}
  \E{
 \pr{\sum_{i\in I_{n,u}}Y_{n,i} }^2
 }=\E{\pr{\sum_{i\in I_{n,u}}Z_{i} }^2}
 + 2\E{\sum_{i\in I_{n,u}}Z_{i}\sum_{i'\in I_{n,u}}Z'_{n,i}  }
 +\E{\pr{\sum_{i\in I_{n,u}}Z'_{n,i} }^2}
 \end{equation}
For the first term, we get, by stationarity, that 
\begin{equation}
 \lim_{n\to+\infty}\frac 1n
 \E{\pr{\sum_{i\in I_{n,u}}Z_{i} }^2}=\pr{t_u-t_{u-1}}
 \sum_{i\in\Z}\cov{Z_0}{Z_i}
\end{equation}

The second and third term cannot be treated similarly 
because of the terms $i/n$. Nevertheless, we can use 
the following lemma.

\begin{Lemma}\label{lem:conv_covariance_in}
 Let $\pr{c_j}_{j\in\Z}$ be an absolutely summable sequence of real 
 numbers such that $c_j=c_{-j}$ for all $j$ and let $0\leq a<b\leq 1$. Then 
 \begin{equation}\label{eq:conv_covariance_in}
  \lim_{n\to+\infty}\frac 1n
  \sum_{i,j=\ent{an}+1}^{\ent{bn}} \frac inc_{j-i}
  = \frac{b^2-a^2}2 \sum_{j\in\Z}c_j;
 \end{equation}
 \begin{equation}\label{eq:conv_covariance_injn}
  \lim_{n\to+\infty}\frac 1n
  \sum_{i,j=\ent{an}+1}^{\ent{bn}} \frac in\frac jnc_{j-i}
  =\frac{b^3-a^3}3\sum_{j\in\Z}c_j.
 \end{equation}
\end{Lemma}
\begin{proof}
 For the first convergence, we split the sum according to 
 the values of $j-i$ (between $N_n-1$ and $-N_n+1$ where 
 $N_n=\ent{bn}-\ent{an}$):
 \begin{equation}
  \sum_{i,j=\ent{an}+1}^{\ent{bn}} \frac inc_{j-i}
  =\sum_{k=-N_n+1}^{N_n-1}c_k
  \sum_{i=\ent{an}+1}^{\ent{bn}}\frac in
  \sum_{j=\ent{an}+1}^{\ent{bn}}\mathbf 1\ens{j-i=k}.
 \end{equation}
The sum $\sum_{j=\ent{an}+1}^{\ent{bn}}\mathbf 1\ens{j-i=k}$ 
is $1$ if $\ent{an}+1\leq i+k\leq\ent{bn}$ and zero otherwise; for 
$k\geq 0$, this constraint means 
$\ent{an}+1\leq i\leq\ent{bn}-k$ and for $k<0$, this means 
$\ent{an}+1-k\leq i\leq\ent{bn}$ hence 
\begin{multline}
 \frac 1n
  \sum_{i,j=\ent{an}+1}^{\ent{bn}} \frac inc_{j-i}
  = 
  \sum_{k\in \Z}\frac 1n\mathbf 1\ens{0\leq k\leq N_n-1} 
  c_k\sum_{i=\ent{an}+1}^{\ent{bn}-k}
  \frac in\\
  +\frac 1n\sum_{k\in \Z}\mathbf 1\ens{-N_n+1\leq k\leq -1} 
  c_k\sum_{i=\ent{an}+1-k}^{\ent{bn}}
  \frac in.
\end{multline}
For a fixed $k$, the convergences 
\begin{equation}
 \frac 1n\mathbf 1\ens{0\leq k\leq N_n-1} 
  c_k\sum_{i=\ent{an}+1}^{\ent{bn}-k}
  \frac in\to c_k\frac{b^2-a^2}2
\end{equation}

\begin{equation}
 \frac 1n\sum_{k\in \Z}\mathbf 1\ens{-N_n+1\leq k\leq -1} 
  c_k\sum_{i=\ent{an}+1-k}^{\ent{bn}}
  \frac in\to c_k\frac{b^2-a^2}2
\end{equation}
holds
and these sequenes are bounded in absolute value 
  by $\abs{c_k}$ hence  \eqref{eq:conv_covariance_in} follows by dominated convergence.
  
 We use the same strategy to show \eqref{eq:conv_covariance_injn}. 
 We start from 
 \begin{equation}
  \sum_{i,j=\ent{an}+1}^{\ent{bn}} \frac in\frac jnc_{j-i}
  =\sum_{k=-N_n+1}^{N_n-1}c_k
  \sum_{i=\ent{an}+1}^{\ent{bn}}\frac in
  \sum_{j=\ent{an}+1}^{\ent{bn}}\frac jn\mathbf 1\ens{j-i=k}.
 \end{equation}
and $\sum_{j=\ent{an}+1}^{\ent{bn}}\frac jn\mathbf 1\ens{j-i=k}$ 
is $\pr{i+k}/n$ if $\ent{an}+1\leq i+k\leq\ent{bn}$ and zero otherwise hence 
\begin{multline}
 \frac 1n
  \sum_{i,j=\ent{an}+1}^{\ent{bn}} \frac in\frac jnc_{j-i}
  = 
  \sum_{k\in \Z}\frac 1n\mathbf 1\ens{0\leq k\leq N_n-1} 
  c_k\sum_{i=\ent{an}+1}^{\ent{bn}-k}
  \frac in\frac{i+k}n\\
  +\frac 1n\sum_{k\in \Z}\mathbf 1\ens{-N_n+1\leq k\leq -1} 
  c_k\sum_{i=\ent{an}+1-k}^{\ent{bn}}
  \frac in \frac{i+k}n.
\end{multline}
\end{proof}
By applying Lemma~\ref{lem:conv_covariance_in} and the 
convergences \eqref{eq:conv_covariance_in} to 
$c_i:=\cov{Z_0}{Z_i''}$ and \eqref{eq:conv_covariance_injn} to 
$c_i:=\cov{Z''_0}{Z_i''}$ we finally obtain that 
\begin{multline}
 \lim_{n\to +\infty}\frac 1n\E{
 \pr{\sum_{i\in I_{n,u}}Y_{n,i} }^2
 }= \sum_{u=1}^{d+1}\pr{t_u-t_{u-1}}
 \sum_{i\in \Z}
 \cov{Z_0}{Z_i}+ \sum_{u=1}^{d+1}\frac{ t_u^2-t_{u-1}^2}2
 \sum_{i\in \Z}
 \cov{Z_0}{Z''_i}\\
 + \sum_{u=1}^{d+1}\frac{ t_u^3-t_{u-1}^3 }3
 \sum_{i\in \Z}
 \cov{Z''_0}{Z''_i}=:\sigma^2.
\end{multline}
If $\sigma=0$, we get a weak convergence to $0$, otherwise 
by Theorem~\ref{thm:TLC_mixing_array},
 we get the weak convergence of $\sum_{k,\ell=1}^da_{k,\ell}W'_n\pr{s_\ell,t_k}$ to $\sum_{k,\ell=1}^da_{k,\ell}
 N_{k,\ell}$, where $\pr{N_{k,\ell}}_{k,\ell=1}^d$ is 
 a Gaussian vector such that for $k\leq k'$, 
  $\operatorname{Cov}\pr{N_{k,\ell},N_{k',\ell'}}= 
 \operatorname{Cov}\pr{W\pr{s_\ell,t_k},
 W\pr{s_{\ell'},t_{k'}}}$, where
 
 \begin{multline}
\operatorname{Cov}\pr{W\pr{s,t},W\pr{s',t'}}=
t\pr{1-t}\pr{1-t'}C_{1,1}\pr{s,s'} -2t\pr{1-t}t'\pr{1-t'}
C_1^a\pr{s,s'}\\
-2t^2\pr{1-t}\pr{1-t'}C_1^a\pr{s',s}+4t^2\pr{1-t}t'\pr{1-t'}
C^a\pr{s,s'}\\
+\pr{t'-t}t\pr{1-t'}C_{2,1}\pr{s,s'}-2\pr{t'-t}tt'\pr{1-t'}
C^a_2\pr{s,s'}\\
-2\pr{t'-t}t\pr{1-t}\pr{1-t'}C_1^a\pr{s',s}+ 
4\pr{t'-t}t\pr{1-t}t'\pr{1-t'}C^a\pr{s,s'}\\
+ \pr{1-t'}tt'C_{2,2}\pr{s,s'}-2\pr{1-t'}tt'\pr{1-t'}
C^a_2\pr{s,s'}\\
-2\pr{1-t'}t\pr{1-t}t'C^a_2\pr{s',s}+4\pr{1-t'}
t\pr{1-t}t'\pr{1-t'}C^a\pr{s,s'}\\
-2 t^2\pr{1-t}t'\pr{1-t'}C^a_1\pr{s,s'}
-2\pr{1-t'}t^2tt'\pr{1-t'}C^a_2\pr{s,s'}\\
+4t\pr{1-t}t'\pr{1-t'}C^a\pr{s,s'}+\frac 43
t\pr{1-t}t'\pr{1-t'}C^a_{s,s'}.
 \end{multline}
Simplifying this expression leads to the covariance 
mentioned in \eqref{eq:expre_cov_oper_centrage_emp}.

 \subsubsection{Convergence of the linear part}
 
 We also use Theorem~\ref{thm:Davydov_Zitikis}. 
 The convergence of $\pr{W'_n}_{n\geq 1}$ also 
 hold to a process having continuous paths. 
 Define 
\begin{equation}
 \xi_n\pr{s,t}:=W'_n\pr{-R+2Rs,t}, s,t\in [0,1],
\end{equation}
where $W'_n$ is defined by \eqref{eq:definition_de_W'n}.
Observe that 
\begin{multline}
\xi_n\pr{s,t}=W_n\pr{-R+2Rs,t}\\
- \frac{[nt]\pr{n-[nt]}}{n^{3/2}}\frac1{{{n}\choose{2}}} \sum_{i=1}^{n-1}\pr{n-i}
 \pr{h_{1,-R+2Rs}\pr{X_i}-\E{h_{1,-R+2Rs}\pr{X_i}}}
\\-  \frac{[nt]\pr{n-[nt]}}{n^{3/2}}\frac1{{{n}\choose{2}}}
\sum_{j=2}^n\pr{j-1}\pr{ 
h_{2,-R+2Rs}\pr{X_j}-\E{h_{2,-R+2Rs}\pr{X_j}}   } 
.
\end{multline}
Now, it suffices to prove that 
\begin{equation}
 \xi_{n,1}\pr{s,t}:= \frac{[nt]\pr{n-[nt]}}{n^{3/2}}\frac1{{{n}\choose{2}}} \sum_{i=1}^{n-1}\pr{n-i}
 \pr{ h_{1,-R+2Rs}\pr{X_i}-\E{h_{1,-R+2Rs}\pr{X_i}}}
 \end{equation}
and 
\begin{equation}
 \xi_{n,2}\pr{s,t}:= \frac{[nt]\pr{n-[nt]}}{n^{3/2}}\frac1{{{n}\choose{2}}} 
\sum_{j=2}^n\pr{j-1}\pr{h_{2,-R+2Rs}\pr{X_j}-
\E{h_{2,-R+2Rs}\pr{X_j}}}
\end{equation}
satisfy the conditions \ref{itm:diff_dec}, 
\ref{itm:cond_1DZ} and \ref{itm:cond_2DZ} of Theorem~\ref{thm:Davydov_Zitikis}.
Since the treatment of $\xi_{n,1}$ is completely analoguous 
to that of $\xi_{n,2}$, we will do it only for the latter. 
By writing $[nt]\pr{n-[nt]}$ and $h_{2,-R+2Rs}\pr{X_j}$ 
as a difference of non-decreasing function, we can see that condition 
\ref{itm:diff_dec} of Theorem~\ref{thm:Davydov_Zitikis} is satisfied with 
\begin{equation}
 \xi^*_{n,2}\pr{s,t}
 =\frac{[nt] }{n^{1/2}}\frac1{{{n}\choose{2}}} 
\sum_{j=2}^n\pr{j-1}F_{2, s}\pr{X_j}
+\frac{[nt]^2 }{n^{3/2}}\frac1{{{n}\choose{2}}} 
\sum_{j=2}^n\pr{j-1}\E{F_{2, s}\pr{X_j}},
\end{equation}
where 
\begin{equation}
 F_{ 2, s}\pr{v}=\PP\ens{g\pr{X_1,v}\leq -R+2Rs}.
\end{equation}
To check condition~\ref{itm:cond_1DZ} of Theorem~\ref{thm:Davydov_Zitikis}, we rewrite $\sum_{j=2}^n\pr{j-1}F_{2, s}\pr{X_j}$ in 
terms of partial sums of $F_{2, s}\pr{X_j}$, 
use triangle inequality for the $\mathbb L^{2p}$-norm, 
then we apply Proposition~\ref{prop:moment_p_mixing}.

Similar estimates as those who led to \eqref{eq:final_step_linear_part} give
\begin{equation} 
\max_{ 0\leq j_1,j_2\leq n  }    \abs{
\xi_{n,2}^*\pr{ \frac{j_1}n,\frac{j_2+1}n    } -
\xi_{n,2}^*\pr{\frac{j_1}n,\frac{j_2}n    }
}\leq 32n^{-1/2}.
\end{equation}

 \subsubsection{Negligibility of the degenerated part}
 
In view of \eqref{eq:definition_de_Rn} and 
\eqref{eq:definition_de_R'n}, it suffices to prove that 
\begin{equation}
 \sup_{s\in [-R,R]}
 \sup_{0 \leq t\leq 1}\frac{[nt]\pr{n-[nt]}}{n^{3/2}}
 {{n}\choose{2}}^{-1}
 \abs{\sum_{1\leq i<j\leq n}\pr{h_{3,s}\pr{X_i}-
 \E{h_{3,s}\pr{X_i}}}}\to 0
 \mbox{ in probability}.
\end{equation}
Bounding $\frac{[nt]\pr{n-[nt]}}{n^{3/2}}
 {{n}\choose{2}}^{-1}$ by $2n^{2-3/2-2}$, it suffices 
 to show that for all positive $\varepsilon$,
 \begin{equation}
  \PP\ens{\sup_{s\in [-R,R]}
 \frac{1}{n^{3/2}}
 \abs{\sum_{1\leq i<j\leq n}\pr{h_{3,s}\pr{X_i}-
 \E{h_{3,s}\pr{X_i}}}}>\varepsilon }
 \to 0.
 \end{equation}

This is done in the same way as in the proof of 
Proposition~\ref{prop:neg_terme_deg}: we cut 
the interval $[-R,R]$ into intervals of length 
$\delta_n=n^{-3/4}$ and do the same 
estimate. This ends the proof of 
Theorem~\ref{thm:convergence_two_sample_emp_Ustat_prob_replaced}.

\subsection{Proof of Corollary~\ref{cor:application}}

We would like to use directly the continuous mapping theorem. However, 
Theorems~\ref{thm:convergence_two_sample_emp_Ustat} and 
\ref{thm:convergence_two_sample_emp_Ustat_prob_replaced} only give the 
convergence in distribution on $D\pr{[-R,R]\times [0,1]}$ which leads to the 
convergence in \eqref{eq:conv_fonctionnelle_pour_en} and \eqref{eq:conv_fonctionnelle_pour_e'n} 
where the integral over $\R$ are replaced by integrals over $[-R,R]$. Then we show that the 
contribution of the integrals on $\R\setminus [-R,R]$ is negligible. 
    
More formally, we will use the Theorem~4.2 in \cite{MR0233396}.
\begin{Proposition}\label{prop:prop_BIll}
Let $\pr{Y_{n}^{\pr{R}}}_{n,R\geq 1}$ be a family of random variable and let $\pr{Y_n}_{n\geq 1}$ and 
$\pr{Z_R}_{R\geq 1}$  be a sequence of random 
variables such that 
\begin{enumerate}
\item\label{itm:item_1_prop_Bill} for all $R\geq 1$, the sequence $\pr{Y_{n }^{\pr{R}}}_{n\geq 1}$ converges in distribution to 
a random variable $Z_R$;
\item\label{itm:item_2_prop_Bill} the sequence $\pr{Z_R}_{R\geq 1}$ converges in distribution to a random variable $Z$ and 
\item\label{itm:item_3_prop_Bill} for all positive $\varepsilon$, 
\begin{equation}
\lim_{R\to +\infty}\limsup_{n\to +\infty}\PP\ens{\abs{Y_{n}^{\pr{R}}-Y_n}>\eps}=0.
\end{equation}
\end{enumerate}
Then the sequence $\pr{Y_n}_{n\geq 1}$ converges in distribution to $Z$.

\end{Proposition}

In order to prove \eqref{eq:conv_fonctionnelle_pour_en} (respectively \eqref{eq:conv_fonctionnelle_pour_e'n}), we apply Proposition~\ref{prop:prop_BIll} to 
\begin{equation}
 Y_{n}^{\pr{R}}=\sup_{0\leq t\leq 1}\int_{[-R,R]}e_n\pr{s,t}^2d\mu\pr{s},  Y_n
=\sup_{0\leq t\leq 1}\int_{\R}e_n\pr{s,t}^2d\mu\pr{s} 
\end{equation}
\begin{equation}
 Z_R=\sup_{0\leq t\leq 1}\int_{[-R,R]}W\pr{s,t}^2d\mu\pr{s}
\end{equation}
 
(respectively 
\begin{equation}
 Y_{n}^{\pr{R}}=\sup_{0\leq t\leq 1}\int_{[-R,R]}e'_n\pr{s,t}^2d\mu\pr{s} , Y_n
=\sup_{0\leq t\leq 1}\int_{\R}e'_n\pr{s,t}^2d\mu\pr{s} ,
\end{equation}
\begin{equation}
 Z_R=\sup_{0\leq t\leq 1}\int_{[-R,R]}W'\pr{s,t}^2d\mu\pr{s}).
\end{equation}

Assumption~\ref{itm:item_1_prop_Bill} follows from the continuous mapping theorem applied with the functional 
\begin{equation}
  \Phi_R\pr{f}=\sup_{0\leq t\leq 1}\int_{[-R,R]}
  f\pr{s,t}^2d\mu\pr{s}.
 \end{equation}
In order to show  that assumptions~\ref{itm:item_2_prop_Bill} and \ref{itm:item_3_prop_Bill} hold,  
it suffices to show  that 
\begin{equation}\label{eq:inegalite_suffisante_pour_Billingsley_en}
\lim_{R\to +\infty}\sup_{n\geq 1}\int_{\R\setminus [-R,R]}\E{\sup_{0\leq t\leq 1}e_n\pr{s,t}^2}d\mu\pr{s} =0
\end{equation}
and 
\begin{equation}\label{eq:inegalite_suffisante_pour_Billingsley_e'n}
\lim_{R\to +\infty}\sup_{n\geq 1}\int_{\R\setminus [-R,R]}\E{\sup_{0\leq t\leq 1}e'_n\pr{s,t}^2}d\mu\pr{s} =0.
\end{equation}
Indeed, for assumption~\ref{itm:item_2_prop_Bill}, 
\begin{equation}    
\abs{Z_R-Z}\leq  \int_{\R\setminus [-R,R]} \sup_{0\leq t\leq 1}W\pr{s,t}^2 d\mu\pr{s}
\end{equation}
hence 
\begin{equation}    
\E{\abs{Z_R-Z}}\leq  \int_{\R\setminus [-R,R]} \E{\sup_{0\leq t\leq 1}W\pr{s,t}^2} d\mu\pr{s}.
\end{equation}
Using the weak convergence of $\pr{e_n\pr{s,t}}_{n\geq 1}$ to $W\pr{s,t}$ on $D\pr{[0,1]}$ for 
a fixed $s$ and then Fatou's lemma, we derive that 
\begin{equation}    
\E{\abs{Z_R-Z}}\leq  \liminf_{n\to +\infty}\int_{\R\setminus [-R,R]} \E{\sup_{0\leq t\leq 1}e_n\pr{s,t}^2} d\mu\pr{s}.
\end{equation}
Moreover, 
\begin{equation}    
\E{\abs{Y_{n}^{\pr{R}}-Y_n}}\leq  \int_{\R\setminus [-R,R]} \E{\sup_{0\leq t\leq 1}e_n\pr{s,t}^2} d\mu\pr{s}
\leq \sup_{n\geq 1} \int_{\R\setminus [-R,R]} \E{\sup_{0\leq t\leq 1}e_n\pr{s,t}^2} d\mu\pr{s},
\end{equation}
and similar inequalities hold where $e_n$ and $W$ are replaced by $e'_n$ and $W$. 

Now, using the fact that $\mu$ is a finite measure, it suffices to find uniform bounds in $n$ and $s$ for 
$\E{\sup_{0\leq t\leq 1}e_n\pr{s,t}^2}$ and $\E{\sup_{0\leq t\leq 1}e'_n\pr{s,t}^2}$, namely, 
\begin{equation}
\sup_{n\geq 1}\sup_{s\in \R}\pr{\E{\sup_{0\leq t\leq 1}e_n\pr{s,t}^2}+\E{\sup_{0\leq t\leq 1}e'_n\pr{s,t}^2}}<+\infty.
\end{equation}
In view of the decompositions \eqref{eq:decomposition_de_e_n} and \eqref{eq:decomposition_de_e'_n} and 
due to the fact that the dependence in $t$ lies in $[nt]$, which is an integer between $0$ and $n$, it suffices 
to show that 
\begin{equation}\label{eq:neglig_Wn}
\sup_{n\geq 1}\sup_{s\in \R}\E{\max_{1\leq k\leq n}W_n\pr{s,\frac{k}{n}}^2           }<+\infty,
\end{equation}
\begin{equation}\label{eq:neglig_W'n}
\sup_{n\geq 1}\sup_{s\in \R}\E{\max_{1\leq k\leq n}W'_n\pr{s,\frac{k}{n}}^2           }<+\infty,
\end{equation}
\begin{equation}\label{eq:neglig_Rn}
\sup_{n\geq 1}\sup_{s\in \R}\E{\max_{1\leq k\leq n}R_n\pr{s,\frac{k}{n}}^2           }<+\infty,
\end{equation}
\begin{equation}\label{eq:neglig_R'n}
\sup_{n\geq 1}\sup_{s\in \R}\E{\max_{1\leq k\leq n}R'_n\pr{s,\frac{k}{n}}^2           }<+\infty,
\end{equation}
where $W_n$, $W'_n$, $R_n$ and $R'_n$ are defined respectively by \eqref{eq:definition_de_Wn}, \eqref{eq:definition_de_W'n}, 
\eqref{eq:definition_de_Rn} and \eqref{eq:definition_de_R'n}.

Let us show \eqref{eq:neglig_Wn}. Observe that 
\begin{multline}
\max_{1\leq k\leq n}W_n\pr{s,\frac{k}{n}}^2\leq 2\pr{\frac{1}{n^{1/2}}\max_{1\leq k\leq n}
\abs{\sum_{i=1}^{k} 
\pr{ h_{1,s}\pr{X_i}-\E{h_{1,s}\pr{X_i}}  }}}^2\\
+2\pr{\frac{ 1}{n^{1/2}} \max_{1\leq k\leq n}\abs{\sum_{j=k+1}^n
 \pr{h_{2,s}\pr{X_j}-\E{h_{2,s}\pr{X_j}}  }  } }^2
\end{multline}
hence using Theorem~\ref{thm:ineg_maximale_moment_2} two times 
gives \eqref{eq:neglig_Wn}. 
 
Let us show \eqref{eq:neglig_W'n}. In view of \eqref{eq:definition_de_W'n}, the equality 
\begin{multline}
W'_n\pr{s,t}= W_n\pr{s,t}-\frac{[nt]\pr{n-[nt]}}{n^{3/2}}\frac1{{{n}\choose{2}}} \sum_{i=1}^{n-1}\pr{n-i}\pr{h_{1,s}\pr{X_i}-
-\E{h_{1,s}\pr{X_i}} }
\\-\frac{[nt]\pr{n-[nt]}}{n^{3/2}}\frac1{{{n}\choose{2}}}
\sum_{j=2}^n\pr{j-1}\pr{h_{2,s}\pr{X_j}-\E{h_{2,s}\pr{X_j}}}
\end{multline}
holds hence it suffices to show that 
\begin{equation*}
\sup_{n\geq 1}\sup_{s\in \R}\frac 1{n^3} \E{\pr{\sum_{i=1}^{n-1}\pr{n-i}\pr{h_{1,s}\pr{X_i}-\E{h_{1,s}\pr{X_i}}}+
\sum_{j=2}^n\pr{j-1}\pr{h_{2,s}\pr{X_j}-\E{h_{2,s}\pr{X_j}}} }^2          }<+\infty.
\end{equation*}
This follows from a rewriting of the sums in terms of partial sums 
of $h_{1,s}\pr{X_i}$ and $h_{2,s}\pr{X_i}$ and an application of 
Theorem~\ref{thm:ineg_maximale_moment_2}.

 Let us show \eqref{eq:neglig_Rn}. Letting $h_{i,j}:=h_{3,s}\pr{X_i,X_j}
 -\E{h_{3,s}\pr{X_i,X_j}}$, we get in view of 
 \eqref{eq:dec_part_deg_mart_rev_mart} that 
\begin{equation} 
\E{\max_{1\leq k\leq n}R_n\pr{s,\frac{k}{n}}^2}= \E{\frac 2{n^{3}}\max_{1\leq k\leq n}
\pr{ \sum_{1\leq i<\ell\leq k}h_{i,\ell}}^2}
+ \E{\frac 2{n^{3 }}\max_{1\leq k\leq n}\pr{ \sum_{\ell=1}^{k }\sum_{j=\ell+1}^nh_{\ell ,j} }^2}    .
\end{equation}
Boundedness follows from Lemma~\ref{lem:moment_ineg_Ustats}.

 Let us show \eqref{eq:neglig_R'n}. Noticing that 
\begin{equation}
R'_n\pr{s,t}= R_n\pr{s,t}-\frac 1{n^{3/2}}\frac{[nt]\pr{n-[nt]}}{{{n}\choose{2}}}\sum_{1\leq i<j\leq n}\pr{ h_{3,s}\pr{X_i,X_j}- 
\E{h_{3,s}\pr{X_i,X_j}}
},
\end{equation}
it suffices to show, in view of \eqref{eq:neglig_Rn}, that 
\begin{equation}
\sup_{n\geq 1}\sup_{s\in \R}\frac 1{n^3} \E{\pr{\sum_{1\leq i<j\leq n}
\pr{ 
h_{3,s}\pr{X_i,X_j}-\E{h_{3,s}\pr{X_i,X_j}}
}}^2}          <+\infty.
\end{equation}
This can be seen by an other use of 
Lemma~\ref{lem:moment_ineg_Ustats}.

\textbf{Acknowledgement} This research was supported by the 
grant
 DFG Collaborative Research Center SFB 823 `Statistical modelling of nonlinear dynamic processes'.

\begin{appendix}
 \section{Facts on mixing sequences}
  
In this section, we collect the facts on mixing sequences we 
need in the proof. 

\begin{Theorem}[Central limit theorem for row-wise 
 mixing arrays, see \cite{MR2200989}]\label{thm:TLC_mixing_array}

Let $\pr{x_{n,j}}_{n\geq 1,1\leq j\leq n}$ be a 
triangular array of centered random variables. For $n\geq 1$, 
let $\alpha_{n}\pr{k}$ be the $k$ mixing coefficient of the 
sequence $\pr{Y_\ell}_{\ell\geq 1}$, where $Y_\ell=0$ if 
$\ell\leq 0$ or $\ell\geq n+1$ and $Y_\ell=x_{n,\ell}$ 
for $1\leq \ell\leq n$. Let $S_n:=\sum_{j=1}^nx_{n,j}$ and 
suppose that the following conditions hold:
\begin{enumerate}
 \item there exists a constant $M$ such that  
 $\sup_{n\geq 1}\max_{1\leq j\leq n} \abs{x_{n,j}}\leq M$ almost 
 surely;
 \item $\lim_{n\to+\infty}n^{-1}\operatorname{Var}\pr{S_n}=
 \sigma^2>0$;
 \item there exists a sequence $\pr{a_k}_{k\geq 1}$ such that 
 $\alpha_{n}\pr{k}\leq a_k$ for all $n$ and all $k$ and 
 for some $r>0$, 
 \begin{equation}
  \sum_{k\geq 1}k^ra_k<\infty.
 \end{equation}

\end{enumerate}
Then $\pr{n^{-1/2}S_n}_{n\geq 1}$ converges in distribution 
to a centered normal random variable with variance 
$\sigma^2$.

 \end{Theorem}

We will also need the following covariance inequality, 
due to Ibragimov \cite{MR0148125}.

\begin{Proposition}\label{prop:cov_inegalite}
 Let $X$ and $Y$ be two bounded random variables. Then 
 \begin{equation}
  \cov{X}{Y}\leq 2\alpha\pr{\sigma\pr{X},\sigma\pr{Y}}
  \norm{X}_\infty\norm{Y}_\infty.
 \end{equation}

\end{Proposition}

In order to control partial sums of an $\alpha$-mixing sequence, 
we need the following maximal inequality 
(see Theorem~3. 1 in \cite{MR2117923}).

\begin{Theorem}\label{thm:ineg_maximale_moment_2}
 Let $\pr{X_i}_{i\geq 1}$ be a centered sequence of random 
 variables bounded by $M$. Then 
 \begin{equation}
  \E{\max_{1\leq k\leq n}\pr{\sum_{i=1}^kX_i}^2}
  \leq 16 M^2 n\sum_{k\geq 0}\alpha\pr{k}.
 \end{equation}

\end{Theorem}

We need the following moment inequality for mixing sequences, in 
the spirit of Rosenthal's inequality \cite{MR0271721}.

\begin{Proposition}[Theorem~2.5 in \cite{MR2117923}]
\label{prop:moment_p_mixing}
 Let $p>1$ and let $\pr{X_i}_{i\geq 1}$ be a strictly stationary 
 sequence of real valued centered 
 random variables bounded by $M$. Then 
 \begin{equation}
  \E{\abs{S_n}^{2p}}\leq \pr{8np}^p
    \int_0^1\pr{\alpha^{-1}
  \pr{u}}^{p } \mathrm du\leq 
  \pr{8np}^p\sum_k k^{p}\alpha\pr{k}
 \end{equation}
where  
\begin{equation}
 \alpha^{-1}\pr{u}=\operatorname{Card}\ens{k\geq 1, \alpha\pr{k}
 \leq u}, u\in [0,1].
\end{equation}

\end{Proposition}

The treatment of the degenerated part requires the following 
moment inequality for a degenerated $U$-statistic, which is 
Lemma~2.4 in \cite{MR2571765}. It was done in the case 
of a symmetric kernel, but a careful reading of the proof 
shows that it also works in the non-symmetric case.

\begin{Lemma} \label{lem:moment_ineg_Ustats}
Let $\pr{X_i}_{i\geq 1}$ be a 
strictly stationary sequence and let 
$h\colon \R^2\to\R$ be a measurable function bounded by 
$M$ and such that for all $x\in \R$, $\E{h\pr{X_1,x}}= 
\E{h\pr{x,X_1}}=0$. Suppose also that 
$\sum_{k\geq 1}k\beta\pr{k}$ converges.
Then for $n\geq 2$, the following 
inequality holds:
\begin{equation}
\E{\max_{2\leq k\leq n}
\pr{\sum_{1\leq i<j\leq k}h\pr{X_i,X_j}}^2}
\leq CM^2 n^2\log n,
\end{equation}
where $C$ depends only on $\pr{\beta\pr{k}}_{k\geq 1}$.
\end{Lemma}

\end{appendix}

\def\polhk\#1{\setbox0=\hbox{\#1}{{\o}oalign{\hidewidth
  \lower1.5ex\hbox{`}\hidewidth\crcr\unhbox0}}}\def\cprime{$'$}
  \def\polhk#1{\setbox0=\hbox{#1}{\ooalign{\hidewidth
  \lower1.5ex\hbox{`}\hidewidth\crcr\unhbox0}}} \def\cprime{$'$}
\providecommand{\bysame}{\leavevmode\hbox to3em{\hrulefill}\thinspace}
\providecommand{\MR}{\relax\ifhmode\unskip\space\fi MR }
\providecommand{\MRhref}[2]{%
  \href{http://www.ams.org/mathscinet-getitem?mr=#1}{#2}
}
\providecommand{\href}[2]{#2}

\end{document}